\documentclass[11pt]{amsart}
\usepackage[margin=2.5cm]{geometry}
\RequirePackage{amsmath,amssymb,amsthm, amscd, comment,mathtools}

\usepackage[all]{xy}

\usepackage[T1]{fontenc}
\usepackage{lmodern}
\usepackage{graphicx}
\usepackage{tikz-cd}
\allowdisplaybreaks[1]
\usepackage[pagebackref,breaklinks,colorlinks,linkcolor=red,anchorcolor=red,citecolor=blue]{hyperref}

\calclayout

\newtheorem{theorem}[subsection]{Theorem}
\newtheorem{proposition}[subsection]{Proposition}
\newtheorem{corollary}[subsection]{Corollary}
\newtheorem{lemma}[subsection]{Lemma}
\newtheorem{conjecture}[subsection]{Conjecture}

\theoremstyle{definition}

\newtheorem{definition}[subsection]{Definition}
\newtheorem{remark}[subsection]{Remark}

\newtheorem{question}[subsection]{Question}

\numberwithin{equation}{subsection}

\newcommand{\A}{\mathbb A}
\newcommand{\Spec}{\operatorname{Spec}}

\newcommand{\cyc}{\mathrm{cyc}}
\newcommand{\spcy}{\operatorname{sp}^{\cyc}}
\newcommand{\spet}{\operatorname{sp}^{\mathrm{\acute{e}t}}}
\newcommand{\Dctf}{D_{\mathrm{ctf}}}

\newcommand{\cC}{\widetilde C}
\newcommand{\VV}{\mathbb{V}}
\makeatletter

\begin{document}

\title[non-acyclicity classes]{Characteristic cycles and non-acyclicity classes for constructible \'etale sheaves}

\author[Kaicheng Bao]{Kaicheng Bao}
\author[Jiangnan Xiong]{Jiangnan Xiong}
\author[Enlin Yang]{Enlin Yang}

\address{School of Mathematical Sciences, Peking University, No.5 Yiheyuan Road Haidian District.,
Beijing, 100871, P.R. China.}
\email{bkc2113@stu.pku.edu.cn}

\address{School of Mathematical Sciences, Peking University, No.5 Yiheyuan Road Haidian District.,
	Beijing, 100871, P.R. China.}
\email{xiongjiangnan@stu.pku.edu.cn}

\address{School of Mathematical Sciences, Capital Normal University, No. 105 XiSanHuan North Road, Beijing 100048, P. R. China}
\email{yangenlin@cnu.edu.cn}




\thanks{\today}

\dedicatory{To the memory of Professor Linsheng Yin (1963--2015)}

\begin{abstract}
Let $f:X\to Y$ be a morphism between smooth schemes over a perfect field $k$. Let $\mathcal F$ be a constructible sheaf on $X$ and $Z$ a closed subscheme of $X$ such that $f$ is $SS(\mathcal F)$-transversal outside $Z$. We construct a class supported on the non-transversal locus $Z$ by using the characteristic cycle $CC(\mathcal F)$ defined by T. Saito. This class is a geometric counterpart of the non-acyclicity class introduced by the third author and Zhao in \cite{YZ22}. Under certain conditions, the formation of this class is compatible with base change and proper pushforward. It also satisfies the Milnor formula and a conductor formula. 

We conjecture that the image of this geometric class under the cycle class map is the non-acyclicity class. This conjecture can be viewed as a (relative version of) Milnor type formula for non-isolated singularities. We prove this conjecture when \(k\) is algebraically closed, \(X\) is quasi-affine and \(Y=\mathbb A^1_k\). 

\end{abstract}

\maketitle
\tableofcontents
\section{Introduction}
\subsection{}Let $k$ be a perfect field of characteristic $p>0$ and $S=\Spec k$. 
Let $\Lambda$ be a finite field of characteristic $\ell\neq p$.
Let $X$ be a smooth scheme over $S$ and $f: X\to Y$ a flat morphism of finite type to a smooth curve $Y$ over $S$. If $f$ has an isolated singularity at a closed point $x_0\in |X|$, there is an invariant $\mu(X/Y,x_0)$ supported on $x_0$, called the Milnor number. The Milnor formula \cite[Th\'eor\`eme 2.4]{Del72} proved by Deligne says that the Milnor number is related to the total dimension at $x_0$ of the vanishing cycles $R\Phi(\Lambda,f)$ of $f$ for the constant sheaf $\Lambda$, i.e.,
\begin{equation}\label{eq:DelMil}
(-1)^n\mu(X/Y,x_0)=-\operatorname{dimtot}R\Phi_{\overline x_0}(\Lambda,f),
\end{equation}
where $n={\rm dim}X$ and $\operatorname{dimtot}={\rm dim}+{\rm Sw}$ denotes the total dimension.
Later in \cite{Del11}, Deligne proposed a Milnor formula for any constructible sheaf $\mathcal F$ of $\Lambda$-modules on $X$, which is realized and proved by Saito  in \cite{Sai17a}. If $x_0\in |X|$ is  at most an isolated characteristic point of $f$ with respect to the singular support of $\mathcal F$, then Saito's theorem \cite[Theorem 5.9]{Sai17a} says
\begin{equation}\label{introSaiM}
(CC(\mathcal F), df)_{T^{\ast} X,x_0}=-\operatorname{dimtot}R\Phi_{\overline x_0}(\mathcal F,f),
\end{equation}
where $CC(\mathcal F)$ is the characteristic cycle of $\mathcal F$. Now we propose the following question:
\begin{question}\label{Que1}
	Is there a Milnor type formula for non-isolated singular/characteristic points?
\end{question}
\subsection{}
If $f$ is a projective flat morphism and if $f$ is smooth outside $f^{-1}(y)$ for a closed point $y$ of the curve $Y$, then the conductor formula of Bloch (cf. \cite[Theorem 2.2.3 and Corollary 2.2.4]{Sai21})
\begin{equation}\label{eq:BlochCond}
-a_y(Rf_\ast\Lambda)=(-1)^{n}(X,X)_{T^{\ast} X,X_y}=(-1)^{n}{\rm deg}c_{n,X_y}^X(\Omega^1_{X/Y})\cap[X]
\end{equation}
gives a partial answer to Question \ref{Que1}.
We view \eqref{eq:DelMil}, \eqref{introSaiM} and \eqref{eq:BlochCond} in the form
\begin{equation}\label{eq:display-135}
{\rm deg}({\rm geometric~class~on~singular~locus})={\rm deg}({\rm cohomology~class~on~singular~locus}).
\end{equation}
We expect the equality holds without taking degree, i.e., 
\begin{equation}\label{eq:display-139}
{\rm cl}({\rm geometric~class})={\rm cohomology~class},
\end{equation}
where ${\rm cl}$ is the cycle class map.
In the paper \cite{YZ22}, the third author and Yigeng Zhao introduced a (cohomological) non-acyclicity class which is supported on the non-acyclicity locus. Let $Y$ be a smooth scheme over $k$ and $X\to Y$ a separated morphism over $k$.
Let $Z\subseteq X$ be a closed subscheme and $\mathcal F\in D_{\rm ctf}(X,\Lambda)$ such that $X\setminus Z\to Y$ is universally locally acyclic relatively to $\mathcal F|_{X\setminus Z}$. Then the cohomological non-acyclicity class
$\widetilde{C}^Z_{X/Y/k}(\mathcal F)$ is a class in $H^0_Z(X,\mathcal K_{X/Y/k})$, where $\mathcal K_{X/Y/k}$ sits in a distinguished triangle
\begin{equation}\label{intro:KXYK}
\mathcal K_{X/Y}\to \mathcal K_{X/k}\to \mathcal K_{X/Y/k}\xrightarrow{+1}.
\end{equation}
In this paper, we construct the geometric counterpart of the non-acyclicity class
$\widetilde{C}^Z_{X/Y/k}(\mathcal F)$. More precisely, when $X\to Y$ is a morphism between smooth schemes of pure dimensions over $k$ such that $X\to Y$ is $SS(\mathcal F)$-transversal outside $Z$, then we construct a
 class $cc^Z_{X/Y/k}(\mathcal F)\in {\rm CH}_{0}(Z)$ (cf. \eqref{eq:defGNA}), called the geometric non-acyclicity class of $\mathcal F$.
Explicitly, we put $r=\dim Y>0$.
Let $q:\mathbb{P}(T^*Y\times_YX)\to X$ be the projective bundle, with the
tautological line bundle $\mathcal O(-1)$ over it. Let $\overline{df}$ be the
section of $q^*T^*X\otimes\mathcal O(1)$ induced by $df:T^*Y\times_Y X\rightarrow T^*X$.
Write $CC(\mathcal F)_{\mathcal O(1)}$ for the conical cycle locally
identified with $q^*CC(\mathcal F)$ after trivializing $\mathcal O(1)$, cf. section~\ref{subsec:defGNA}.
Define
\begin{equation}\label{eq:intro-defGNA}
cc^Z_{X/Y/k}(\mathcal F)
:= q_{Z*}\left[
c_{r-1}\left(\frac{q^*f^*T^*(Y/k)}{\mathcal O(-1)}\right)
\cap\overline{df}^{\,!}CC(\mathcal F)_{\mathcal O(1)}
\right]\in\mathrm{CH}_0(Z).
\end{equation}
By $SS(\mathcal F)$-transversal condition of $f$ outside $Z$, we may define the 
(geometric)
relative characteristic class
$cc_{X\setminus Z/Y}(\mathcal F)\in \operatorname{CH}_{\dim Y}(X\setminus Z)$
which has a canonical extension $\overline{cc}_{X/Y}(\mathcal F)\in\operatorname{CH}_{\dim Y}(X)$. 
For $\tau:Z\hookrightarrow X$, we have the fibration formula \eqref{eq:relative-characteristic-fibration}
\begin{equation}\label{eq:display-152}
cc_{X/k}(\mathcal F)=c_{\rm dim Y}(f^\ast\Omega_{Y/k}^{1,\vee})\cap \overline{cc}_{X/Y}(\mathcal F)+\tau_*cc^Z_{X/Y/k}(\mathcal F).
\end{equation}
In general, we may replace the $\Spec k$ with general smooth connected scheme $S$, and we have the following properties, cf.~equation~\eqref{eq:relative-characteristic-fibration},
Propositions~\ref{prop:pbGNA}, \ref{prop:pfGNA},
and~\ref{prop:geoNA-affine-line},
Theorem~\ref{thm:SMF},
Corollary~\ref{cor:projective-conductor}:
\begin{theorem}
\label{thm:intro-geometricNA-properties}
Let $k$ be a perfect field of positive characteristic and $\Lambda$
a finite local ring whose residue characteristic is invertible in $k$.
Consider morphisms $f:X\rightarrow Y$, $g:Y\rightarrow S$
of smooth separated schemes of finite type over $k$, all of pure
dimension. Assume that $S$ is connected of dimension $s$ and that
$g$ is smooth of pure relative dimension $r>0$.
Put $h=g\circ f$, and let $\tau:Z\hookrightarrow X$ be a closed immersion.
Let $\mathcal F\in D_{\rm ctf}(X,\Lambda)$ such that $h$ is
$SS(\mathcal F)$-transversal and $f|_{X\setminus Z}$ is
$SS(\mathcal F|_{X\setminus Z})$-transversal.
Then the geometric non-acyclicity class
$cc^Z_{X/Y/S}(\mathcal F)\in{\rm CH}_s(Z)$
has the following properties.

\begin{enumerate}
\item $($Fibration formula$)$
Let
\begin{equation}
    cc_{X/S}(\mathcal F)=(-1)^sdh^!CC(\mathcal F)
\end{equation}
be the geometric characteristic class.
There is a canonical class
$\overline{cc}_{X/Y}(\mathcal F)\in{\rm CH}_{r+s}(X)$
whose restriction to $X\setminus Z$ is
$cc_{(X\setminus Z)/Y}(\mathcal F|_{X\setminus Z})$.
We have
\begin{equation}\label{eq:intro-geometricNA-fibration}
cc_{X/S}(\mathcal F)
=
c_r\bigl((f^*\Omega^1_{Y/S})^\vee\bigr)
\cap\overline{cc}_{X/Y}(\mathcal F)
+\tau_*cc^Z_{X/Y/S}(\mathcal F)
\quad{\rm in}\quad{\rm CH}_s(X).
\end{equation}

\item $($Pullback$)$
Let $\delta:S'\to S$ be a morphism, where $S'$ is smooth and
connected of dimension $s'$ over $k$.
Put $X'=X\times_SS'$, $Y'=Y\times_SS'$ and $Z'=Z\times_SS'$,
and let $i_X:X'\to X$ and $f':X'\to Y'$ be the induced morphisms.
Assume that $X'$ and $Y'$ are smooth and of pure dimension,
that $\dim X'-\dim X=s'-s$, and that $i_X$ is properly
$SS(\mathcal F)$-transversal. Then we have
\begin{equation}\label{eq:intro-geometricNA-pullback}
i_X^!cc^Z_{X/Y/S}(\mathcal F)
=
cc^{Z'}_{X'/Y'/S'}(i_X^*\mathcal F)
\quad{\rm in}\quad{\rm CH}_{s'}(Z').
\end{equation}

\item $($Proper pushforward$)$
Let $p:X\to X'$ be a proper morphism over $Y$, where $X'$ is
smooth and of pure dimension over $k$, and let $f':X'\to Y$
satisfy $f=f'\circ p$.
Let $Z'\subseteq X'$ be closed with $Z\subseteq p^{-1}(Z')$.
Assume that $\dim p_\circ SS(\mathcal F)\leq\dim X'$,
where $p_\circ SS(\mathcal F)$ is the image of
$(dp)^{-1}SS(\mathcal F)$ under
$T^*X'\times_{X'}X\to T^*X'$.
Then $g\circ f'$ is $SS(Rp_*\mathcal F)$-transversal and
$f'|_{X'\setminus Z'}$ is
$SS((Rp_*\mathcal F)|_{X'\setminus Z'})$-transversal.
Moreover, we have
\begin{equation}\label{eq:intro-geometricNA-pushforward}
p_*cc^Z_{X/Y/S}(\mathcal F)
=
cc^{Z'}_{X'/Y/S}(Rp_*\mathcal F)
\quad{\rm in}\quad{\rm CH}_s(Z').
\end{equation}

\item $($Milnor formula$)$
Assume $S=\Spec k$, $Y$ is a smooth curve and $Z$ is a finite
set of closed points, then we have
\begin{equation}\label{eq:intro-geometricNA-Milnor}
cc^Z_{X/Y/k}(\mathcal F)
=
-\sum_{x\in Z}
\operatorname{dimtot}R\Phi_{\bar x}(\mathcal F,f)\,[x]
\quad{\rm in}\quad{\rm CH}_0(Z),
\end{equation}
where $R\Phi(\mathcal F,f)$ is the complex of vanishing cycles
and $\operatorname{dimtot}={\rm dim}+{\rm Sw}$ is the total dimension.

\item $($Conductor formula$)$
Assume $S=\Spec k$, $Y$ is a smooth connected curve,
and $f$ is proper.
If $T\subset Y$ is a finite set of closed points and
$Z=f^{-1}(T)$, then we have
\begin{equation}\label{eq:intro-geometricNA-conductor}
f_*cc^Z_{X/Y/k}(\mathcal F)
=
-\sum_{y\in T}a_y(Rf_*\mathcal F)\,[y]
\quad{\rm in}\quad{\rm CH}_0(T),
\end{equation}
where
$a_y(\mathcal G)
={\rm rank}\mathcal G_{\bar\eta}
-{\rm rank}\mathcal G_{\bar y}
+{\rm Sw}_y(\mathcal G)$
is the Artin conductor of
$\mathcal G\in D_{\rm ctf}(Y,\Lambda)$ at $y$, and $\eta$
is the generic point of $Y$.

\end{enumerate}
\end{theorem}

In view of the similarity of $cc^Z_{X/Y/k}(\mathcal F)$ to the cohomological
non-acyclicity class,
it is natural to formulate the following conjecture:
\begin{conjecture}[Conjecture \ref{conj:MFforNI}]\label{conj:intro-MFforNI} We have
	\begin{equation}\label{eq:MFfornonisolated}
\widetilde{C}^Z_{X/Y/k}(\mathcal F)=\widetilde{{\rm cl}}(cc^Z_{X/Y/k}(\mathcal F))\quad{\rm in}\quad H^0_Z(X,\mathcal K_{X/Y/k}),
\end{equation}
	where $\widetilde{{\rm cl}}: {\rm CH}_{0}(Z)\xrightarrow{\rm cl} H^0_Z(X,\mathcal K_{X/k})\xrightarrow{\eqref{intro:KXYK}} H^0_Z(X,\mathcal K_{X/Y/k})$ and ${\rm cl}$ is the cycle class map.
\end{conjecture}
We hope \eqref{eq:MFfornonisolated} gives a formulation of Question \ref{Que1} in some sense.

When $Z=X$, let
$\rho:H^0(X,\mathcal K_{X/k})\to H^0(X,\mathcal K_{X/Y/k})$
be the map induced by \eqref{intro:KXYK} and let $0_X:X\hookrightarrow T^*X$ be the
zero section. The conjecture becomes
\begin{equation}\label{eq:intro-case-ZX}
\rho\bigl({\rm cl}(0_X^!CC(\mathcal F))\bigr)
=\rho\bigl(C_{X/k}(\mathcal F)\bigr),
\end{equation}
which is the image of a conjectured formula by Saito
\cite[Conjecture~6.8(1) and Lemma~6.9(1)]{Sai17a} under $\rho$.
When $f={\rm id}_X$, $X$ is irreducible and
$U:=X\setminus Z\neq \varnothing$, the conjecture becomes
\begin{equation}\label{eq:intro-case-identity}
C_Z^0(\mathcal F)
={\rm cl}\!\left(
0_X^!\bigl(CC(\mathcal F)
-{\rm rank}(\mathcal F|_U)\cdot CC(\Lambda)\bigr)
\right)
\quad\text{in }H_Z^0(X,\mathcal K_{X/k}),
\end{equation}
where $C_Z^0(\mathcal F)$ is the localized characteristic class
of Abbes--Saito \cite[Definition~5.2.1]{AS07}, 
$\operatorname{rank}\mathcal F$ is the generic rank and the difference of characteristic cycles is supported
on $T^*X\times_XZ$.
In fact, by \cite[p.~163, immediately after Theorem~5.1]{YZ22},
\begin{equation}
    \rho_Z:H^0_Z(X,\mathcal K_{X/k})\longrightarrow H^0_Z(X,\mathcal K_{X/X/k})
\end{equation}
is an isomorphism in this case, and we have $C_Z^0(\mathcal F)=\widetilde{C}^Z_{X/X/k}(\mathcal F)$.


In Section~4, we prove Conjecture~\ref{conj:intro-MFforNI} in the quasi-affine case for
morphisms to the affine line over an algebraically closed field.
More precisely, we have the following theorem.

\begin{theorem}[Theorem~\ref{thmqafmil}]
\label{thmqafint}
Let $k$ be an algebraically closed field of characteristic $p>0$
and $\Lambda$ a finite local ring such that the characteristic of
its residue field is invertible in $k$. Let $X$ be a smooth
quasi-affine scheme of finite type over $k$, purely of dimension $d$.
Let $f:X\to\mathbb{A}^1_k=\Spec k[u]$ be a morphism and
$\mathcal{F}\in D_{\rm ctf}(X,\Lambda)$. Put
$Z=(df)^{-1}SS(\mathcal{F})$, where
\begin{equation}
    df=f^*du:X\rightarrow T^*\A^1_k\times_{\A^1_k} X\to T^*X
\end{equation}
is the section defined by the differential of $f$.
Then we have
\begin{equation}\label{eq:auto-001}
\widetilde{C}^{Z}_{X/\mathbb{A}^1_k/k}(\mathcal{F})
=
\widetilde{\operatorname{cl}}
\bigl(df^!CC(\mathcal{F})\bigr)
\quad\text{in }
H^0_Z\bigl(X,\mathcal{K}_{X/\mathbb{A}^1_k/k}\bigr).
\end{equation}
\end{theorem}

The quasi-affineness of $X$ allows us to choose a global function $g$
such that the generic member of the family $f+ug$ has only isolated
characteristic points (Proposition~\ref{qafeq} and Lemma~\ref{genpert}).
After passing to the perfect
closure of the generic point of $\A^1_k$, the desired equality follows from
Saito's Milnor formula \eqref{eq:SMF} and the cohomological Milnor formula \eqref{eq:CMF}.
We then descend this equality to the generic point and specialize
to $u=0$. For this, we prove that the supported cycle class map is
compatible with specialization (Theorem~\ref{spcycommthm}).
Then theorem~\ref{thmqafint} follows from
the specialization formulas for refined Gysin pullbacks and
non-acyclicity classes.
\subsection*{Acknowledgments} 
The authors would like to thank Xiangyu Pan and Yigeng Zhao for helpful discussions.
This work was partially supported by the National Key R{\&}D Program of China (Grant No. 2021YFA1001400), NSFC Grant No. 12522101 and NSFC Grant No. 12271006.
The third author is supported by Beijing Municipal Natural Science Foundation (Grant No. QY23005).
\subsection*{Notation and Conventions}
\begin{enumerate}
	\item Let $S$ be a Noetherian scheme and ${\rm Sch}_{S}$ the category of separated schemes of finite type over $S$. 
	\item Let $\Lambda$ be a Noetherian ring such that $m\Lambda=0$ for some integer $m$ invertible on $S$ unless otherwise stated explicitly.
	\item For any scheme $X\in {\rm Sch}_S$,
	we denote by $D_{\rm ctf}(X,\Lambda)$ the derived category of complexes of $\Lambda$-modules of finite Tor-dimension with constructible cohomology sheaves on $X$.
	\item  For any separated morphism $f:X\to Y$ in ${\rm Sch}_S$, we use the following notations
	\begin{equation}
\mathcal K_{X/Y}=Rf^!\Lambda,\quad D_{X/Y}(-)=R\mathcal Hom(-, \mathcal K_{X/Y}).
\end{equation}
	
	\item For $\mathcal F\in D_{\rm ctf}(X,\Lambda)$ and $\mathcal G\in D_{\rm ctf}(Y,\Lambda)$ on $S$-schemes $X$ and $Y$  respectively, $\mathcal F\boxtimes_S^L\mathcal G$ denotes ${\rm pr}_1^\ast\mathcal F\otimes^L{\rm pr}_2^\ast\mathcal G$ on $X\times_SY$.
	\item To simplify our notation, we omit $R$ or $L$ denoting the derived functors, unless stated explicitly or for $R\mathcal Hom$.
\end{enumerate}

\section{Transversality condition}
We recall the cohomological and geometric transversality conditions, which express a notion of smoothness of a morphism \(f\colon X\to S\) with respect to a constructible sheaf \(\mathcal F\) on \(X\).

\subsection{}\label{def:gtrans}
We recall the (cohomological) transversality condition introduced in \cite[\S2.1]{YZ22}, which is a relative version of the transversality condition studied by Saito \cite[Definition 8.5]{Sai17a}.
Consider the following Cartesian diagram in ${\rm Sch}_S$:
\begin{equation}\label{eq:carGtrans}
\begin{gathered}
		\xymatrix{
			X\ar@{}[rd]|\Box\ar[d]_p\ar[r]^-i&Y\ar[d]^-f\\
			W\ar[r]^-{\delta}&T.
		}
	\end{gathered}
\end{equation}
Let $\mathcal F\in D_{\rm ctf}(Y,\Lambda)$ and $\mathcal G\in D_{\rm ctf}(T,\Lambda)$. 
Let $c_{\delta,f,\mathcal F,\mathcal G}$ be the  composition 
\begin{equation}\label{eq:carGtrans1}
\begin{split}
		c_{\delta,f,\mathcal F,\mathcal G}:i^\ast \mathcal F\otimes^L p^\ast \delta^!\mathcal G&\xrightarrow{id\otimes {\rm b.c}}i^\ast \mathcal F\otimes^L i^! f^\ast\mathcal G \\
		&\xrightarrow{\rm adj}i^!i_!(i^\ast \mathcal F\otimes^L i^! f^\ast\mathcal G)\\
		&\xrightarrow[\simeq]{\rm proj.formula}i^!(\mathcal F\otimes^L i_!i^! f^\ast\mathcal G)\xrightarrow{\rm adj}i^!(\mathcal F\otimes^L f^\ast\mathcal G).
	\end{split}
\end{equation}
We put $c_{\delta,f,\mathcal F}:=c_{\delta,f,\mathcal F,\Lambda}: i^\ast\mathcal F\otimes^L p^\ast\delta^!\Lambda\to i^!\mathcal F$.
If $c_{\delta,f,\mathcal F}$ is an isomorphism, then we say that the morphism $\delta$ is {\it $\mathcal F$-transversal}.
If $c_{i,{\rm id},\mathcal F}$ is an isomorphism, then we say $i$ is $\mathcal F$-transversal.

The following lemma gives an equivalence characterization between transversality condition and (universally) locally acyclicity condition.
\begin{lemma}\label{lem:Trans=ULA}
	Let $f:X\to S$ be a separated morphism of finite type between Noetherian schemes and $\mathcal F\in D_{\rm ctf}(X,\Lambda)$. 
	The following conditions are equivalent:
	\begin{enumerate}
		\item The morphism $f$ is  locally acyclic relatively to $\mathcal F$.
		\item The morphism $f$ is universally locally acyclic relatively to $\mathcal F$.
		\item For any $\mathcal G\in D_{\rm ctf}(X,\Lambda)$, the canonical map
		\begin{equation}\label{eq:display-226}
D_{X/S}(\mathcal G)\boxtimes_S^L \mathcal F\xrightarrow{}R\mathcal Hom_{X\times_SX}({\rm pr}_1^\ast\mathcal G,{\rm pr}_2^!\mathcal F)
\end{equation}
		is an isomorphism in $D_{\rm ctf}(X\times_SX,\Lambda)$, where ${\rm pr}_1:X\times_SX\to X$ and ${\rm pr}_2:X\times_SX\to X$ are the projections.
		\item The canonical map
		\begin{equation}\label{eq:display-231}
D_{X/S}(\mathcal F)\boxtimes_S^L \mathcal F\xrightarrow{}R\mathcal Hom_{X\times_SX}({\rm pr}_1^\ast\mathcal F,{\rm pr}_2^!\mathcal F)
\end{equation}
		is an isomorphism.
		\item For any Cartesian diagram 
		\begin{equation}\label{eq:deltaIbase}
\begin{gathered}
				\xymatrix{ Y\times_SX\ar[d]_-{{\rm pr}_1} \ar[r]^-{{\rm pr}_2} & X  \ar[d]^-f \\
					Y\ar[r]_-{\delta} &S\ar@{}|\Box[ul],
				}
			\end{gathered}
\end{equation}
		such that all schemes are Noetherian and $\delta$ is separated of finite type, the morphism $\delta$ is $\mathcal F$-transversal.
		\item For any Cartesian diagram
		\eqref{eq:deltaIbase} and any $\mathcal G\in D_{\rm ctf}(S,\Lambda)$, the morphism $c_{\delta,f,\mathcal F,\mathcal G}$ is an isomorphism.
		\item For any Cartesian diagram 
		\begin{equation}\label{eq:deltaIbase2}
\begin{gathered}
				\xymatrix{ 
					Y\times_SX\ar[d]_-{{\rm pr}_1} \ar[r]^-{{\rm pr}_2} \ar@{}|\Box[rd]& X'  \ar[d]^-{f'}\ar[r]\ar@{}|\Box[rd]&X\ar[d]^-f \\
					Y\ar[r]_-{\delta} &S'\ar[r]&S,
				}
			\end{gathered}
\end{equation}
        such that all schemes are Noetherian and $\delta$ is separated of finite type, the morphism $\delta$ is $\mathcal F|_{X'}$-transversal.
		\item For any Cartesian diagram
		\eqref{eq:deltaIbase2} and any $\mathcal G\in D_{\rm ctf}(S',\Lambda)$, the morphism $c_{\delta,f',\mathcal F|_{X'},\mathcal G}$ is an isomorphism.
	\end{enumerate}
\end{lemma}
When $S$ is a scheme of finite type over a field $k$, then the equivalence between (2) and (7) follows from \cite[Propositions~2.3(2) and 2.4]{YZ22}. In this case, we may require $Y$ and $S'$ smooth over $k$ in (7).
\begin{proof}
	By a result of Gabber \cite[Corollary 6.6]{LZ19}, (1) and (2) are equivalent.
	The equivalence between (2), (3), and (4) follows from \cite[Proposition~2.5, Lemma~2.14, and Theorem~2.16]{LZ22}.
	By \cite[Proposition~2.3(2)]{YZ22}, (2) implies (6). It is clear that (6) implies (5).
	
	Now we show (5) implies (1).
 By \cite[Theorem 2.16]{LZ22}, it suffices to show that for any  $\mathcal H\in D_{\rm ctf}(Y,\Lambda)$, the canonical morphism 
	\begin{equation}\label{eq:kunneth}
D_{Y/S}(\mathcal H)\boxtimes_S^L \mathcal F\xrightarrow{}R\mathcal Hom({\rm pr}_1^\ast\mathcal H,{\rm pr}_2^!\mathcal F)
\end{equation}
	is an isomorphism. 
    Using a standard d\'evissage argument \cite[Expos\'e~IX, Proposition~2.7]{SGA4}, we may assume $\mathcal H=j_!\Lambda$ for $j: U\to Y$ \'etale with $U$ affine. 
    In this case, \eqref{eq:kunneth} is the following composition
	\begin{equation}\label{eq:display-277}
j_\ast j^\ast\mathcal K_{Y/S}\boxtimes_S^L \mathcal F\xrightarrow{}(j\times{\rm id})_\ast( j^\ast\mathcal K_{Y/S}\boxtimes_S^L \mathcal F)\xrightarrow[\simeq]{\eqref{eq:kun0}} (j\times {\rm id})_\ast (j\times {\rm id})^\ast {\rm pr}_2^!\mathcal F.
\end{equation}
    In the following we show that 
    \begin{equation}\label{eq:dualsff}
j_\ast j^\ast \mathcal K_{Y/S}\boxtimes^L_S \mathcal F\xrightarrow{}(j\times{\rm id})_\ast (j^\ast \mathcal K_{Y/S}\boxtimes^L_S \mathcal F)
\end{equation}
    is an isomorphism. We write $j$ as a composition of a proper morphism and an open immersion.

	Since $\delta$ is $\mathcal F$-transversal, we have an isomorphism
	\begin{equation}\label{eq:kun0}
\mathcal K_{Y/S}\boxtimes_S^L \mathcal F\xrightarrow{\simeq}{\rm pr}_2^!\mathcal F.
\end{equation}
By the projection formula and proper base change, for any proper morphism $g: Y'\to Y$ and $A \in D_{\rm ctf}(Y',\Lambda)$, the following canonical morphism
	\begin{equation}\label{eq:kun0-0}
g_! A\boxtimes_S^L \mathcal F\xrightarrow{\simeq} (g\times {\rm id})_!(A\boxtimes_S^L \mathcal F)
\end{equation}
	is an isomorphism. If $u: Y'\to Y$ is an open immersion with complementary closed immersion $\tau: Z=Y\setminus Y'\to Y$, we have a commutative diagram between distinguished triangles
	\begin{equation}\label{eq:display-297}
\begin{gathered}
			\xymatrix{
				\tau_! \tau^! \mathcal K_{Y/S}\boxtimes_S^L \mathcal F\ar[r]\ar[d]^-\simeq_-{\eqref{eq:kun0-0}} & \mathcal K_{Y/S}\boxtimes_S^L \mathcal F\ar[r]\ar[d]^-\simeq&u_\ast u^\ast \mathcal K_{Y/S}\boxtimes_S^L \mathcal F\ar[d]\ar[r]^-{+1}&\\
				(\tau\times {\rm id})_!(\tau^! \mathcal K_{Y/S}\boxtimes_S^L \mathcal F)\ar[r]\ar[d]^-\simeq_-{\eqref{eq:kun0}}&\mathcal K_{Y/S}\boxtimes_S^L \mathcal F\ar[r]\ar@{=}[d]&(u\times {\rm id})_\ast(u^\ast \mathcal K_{Y/S}\boxtimes_S^L \mathcal F)\ar[d]^-\simeq\ar[r]^-{+1}&\\
				(\tau\times {\rm id})_!(\tau\times {\rm id})^!( \mathcal K_{Y/S}\boxtimes_S^L \mathcal F)\ar[r]&\mathcal K_{Y/S}\boxtimes_S^L \mathcal F\ar[r]&(u\times {\rm id})_\ast(u\times{\rm id})^\ast (\mathcal K_{Y/S}\boxtimes_S^L \mathcal F)\ar[r]^-{+1}&.
			}
		\end{gathered}
\end{equation}
	Thus 
	\begin{equation}\label{eq:openkunnpf}
u_\ast u^\ast \mathcal K_{Y/S}\boxtimes_S^L \mathcal F\to (u\times {\rm id})_\ast(u^\ast \mathcal K_{Y/S}\boxtimes_S^L \mathcal F)
\end{equation}
 is an isomorphism.
 Applying \eqref{eq:openkunnpf} and \eqref{eq:kun0-0}, we obtain that \eqref{eq:dualsff} is an isomorphism, and the conclusion follows.
	
	We show (2) implies (8). By assumption, $f'$ is also universally locally acyclic relatively to $\mathcal F|_{X'}$. Thus by (6), the morphism $c_{\delta,f',\mathcal F|_{X'},\mathcal G}$ is an isomorphism.
	
	Finally, it is clear that (8) implies (7) and (7) implies (5).
\end{proof}

\subsection{}\label{def:transdef00}
Now we recall the geometric transversality condition (cf. \cite[1.2]{Bei16} and \cite[Definitions~3.3, 3.5, and 7.1]{Sai17a}).
Let $X$ be a smooth scheme over a field $k$. Let $C$ be a conical closed subset of $T^{\ast}X$, i.e., a closed subset which is stable under the action of the multiplicative group $\mathbb{G}_m$. 	
We denote by $T^{\ast}_XX\subseteq T^{\ast} X$ the zero section of the cotangent bundle $T^{\ast} X$ of $X$.

	\begin{enumerate}
		\item  Let $h\colon W\to X$ be a morphism from a smooth scheme $W$ over $k$. 
		We say that $h$  is $C$-transversal  if the intersection 
		$(C\times_X W)\cap dh^{-1}(T^{\ast}_WW)$ is contained in the zero-section $ T^{\ast}_XX\times_X W\subseteq T^{\ast} X\times_X W$,
		where $dh\colon T^{\ast}X\times_XW\to T^{\ast}W$ is the canonical map.
		
		If $h$ is $C$-transversal,
		we define $h^\circ C$ to be the image of $C\times_XW$ under the map $dh\colon T^{\ast}X\times_XW\to T^{\ast}W$.
		By \cite[Lemma 3.1]{Sai17a}, $h^\circ C$ is a closed conical subset of $T^{\ast} W$.
	
		\item  Assume that $X$ and $C$ are purely of dimension $d$ and 
		that $W$ is purely of dimension $m$.
		We say that a $C$-transversal map $h\colon W\to X$ is properly $C$-transversal if every irreducible component of $C\times_XW$ is of dimension $m$. 
		
		\item  We say that a morphism $f\colon X\to Y$ to a smooth scheme $Y$ over $k$ is $C$-transversal  if the inverse image $df^{-1}(C)$ is contained in the zero-section $T^{\ast} _YY\times_YX\subseteq T^{\ast} Y\times_Y X$, where $df\colon T^{\ast} Y\times_Y X\rightarrow T^{\ast} X$ is the canonical map.

        If $f$ is proper over the base
$C\cap T_X^*X$ of $C$,
we define $f_\circ C$ to be the image of $(df)^{-1}(C)$ under the projection
${\rm pr}:T^*Y\times_YX\to T^*Y$.
In this case, $f_\circ C$ is also a closed conical subset of $T^*Y$.
	\end{enumerate}
The cohomological transversality condition and geometric transversality condition are related as follows. Let $X$ be a smooth scheme of pure dimension $d$ and $\Lambda$ a finite local ring such that the characteristic $\ell$ of the residue field of $\Lambda$ is invertible in $k$. Let $\mathcal F\in D_{\rm ctf}(X,\Lambda)$. The singular support $SS(\mathcal F)$ defined by Beilinson \cite{Bei16} is a $d$-dimensional conical closed subset of $T^{\ast} X$. We have the following properties:
\begin{enumerate}
	\item $SS(\mathcal F)\cap T^{\ast}_XX={\rm supp}(\mathcal F)$.
	\item Let $f: X\to Y$ be a morphism to a smooth scheme $Y$ over $k$. If $f$ is $SS(\mathcal F)$-transversal, then $f$ is universally locally acyclic relatively to $\mathcal F$.
	\item Let $f: W\to X$ be a morphism from a smooth scheme $W$ over $k$. If $f$ is $SS(\mathcal F)$-transversal, then $f$ is $\mathcal F$-transversal.
\end{enumerate}
Now assume $k$ is a perfect field. By \cite[Theorem 5.9 and Theorem 5.18]{Sai17a}, the characteristic cycle $CC(\mathcal F)$ is the unique $d$-cycle $CC(\mathcal F)$ supported on $SS(\mathcal F)$ with $\mathbb Z$-coefficients such that $CC(\mathcal F)$ satisfies the Milnor formula \eqref{introSaiM} for $\mathcal F$.

\section{Geometric non-acyclicity classes}
\subsection{}\label{subsec:defC}
Let $S$ be a Noetherian scheme and  $\Lambda$ a Noetherian ring such that $m\Lambda=0$ for some integer $m$ invertible on $S$.
Consider the following Cartesian diagram in ${\rm Sch}_S$
\begin{equation}\label{diagXSY}
\begin{gathered}
		\xymatrix{
			X\times_SY\ar[d]_-{\rm pr_2}\ar[r]^-{\rm pr_1}\ar@{}|\Box[rd]&X\ar[d]^-h\\
			Y\ar[r]^-g&S,	
		}
	\end{gathered}
\end{equation}
where ${\rm pr}_1$ and ${\rm pr}_2$ are the projections.
For any $\mathcal F\in D_{\rm ctf}(X,\Lambda)$ and $\mathcal G\in D_{\rm ctf}(Y,\Lambda)$, 
we have canonical morphisms
\begin{equation}\label{eq:cForFBoxK}
\begin{aligned}
&\mathcal F\boxtimes_S^L \mathcal K_{Y/S}={\rm pr}^\ast_1\mathcal F\otimes^L{\rm pr}_2^\ast g^!\Lambda\xrightarrow{c_{g,h,\mathcal F}}{\rm pr}_1^!\mathcal F,
\end{aligned}
\end{equation}
\begin{equation}\label{eq:Kunethiso}
\begin{aligned}
&
	\mathcal F\boxtimes_S^{L}D_{Y/S}(\mathcal G)\rightarrow R\mathcal Hom_{X\times_SY}({\rm pr}^*_2\mathcal G, {\rm pr}^!_1\mathcal F),
\end{aligned}
\end{equation}
where \eqref{eq:Kunethiso} is adjoint to 
\begin{equation}\label{Def-eq:Kunethiso}
\mathcal F\boxtimes_S^{L}(D_{Y/S}(\mathcal G)\otimes^L \mathcal G)\xrightarrow{id\boxtimes{\rm ev}} \mathcal F\boxtimes_S^L \mathcal K_{Y/S}\xrightarrow{\eqref{eq:cForFBoxK}}{\rm pr}_1^!\mathcal F.
\end{equation}
Note that \eqref{eq:cForFBoxK} is a special case of \eqref{eq:Kunethiso} by taking $\mathcal G=\Lambda$. 
If moreover $X\to S$ is universally locally acyclic relatively to  $\mathcal F$, then
\eqref{eq:Kunethiso} is an isomorphism by  \eqref{eq:kunneth}.
For a morphism $c=(c_1,c_2):C\to X\times_S Y$, we have a  canonical isomorphism by \cite[Corollaire 3.1.12.2]{Del73}
\begin{equation}\label{eq:keyisoCshrinkRHOM}
R\mathcal Hom(c^*_2\mathcal G, c^!_1\mathcal F) \xrightarrow{\simeq}c^!R\mathcal Hom({\rm pr}^*_2\mathcal G, {\rm pr}^!_1\mathcal F).
\end{equation}
\subsection{}\label{subsec:notationforXYS}
Consider a commutative diagram in ${\rm Sch}_S$:
\begin{equation}\label{xyd:prop:Gysim:fgh2}
\begin{gathered}
		\xymatrix{
			Z\ar@{^(->}[r]^-\tau&X\ar[rr]^-f\ar[rd]_-h&&Y,\ar[ld]^-g\\
			&&S&
		}
	\end{gathered}
\end{equation}
where $\tau: Z\to X$ is a closed immersion and $g$ is a smooth morphism. 
Let $i: X\times_Y X\to X\times_S X$ be the base change of the diagonal morphism $\delta:Y\to Y\times_SY$:
\begin{equation}\label{eq:propchangebasehomo:1}
\begin{gathered}
		\xymatrix{
			X\ar@/_2pc/[dd]_-{f}\ar@{_(->}[d]_-{\delta_1}\ar@{}[rd]|\Box\ar@{=}[r]&X\ar@{_(->}[d]^-{\delta_0}\\
			X\times_Y X\ar[r]^-{i}\ar[d]_-p \ar[r] & X\times_S X  \ar[d]^-{f\times f} \\
			Y\ar[r]^-\delta&Y\times_SY,\ar@{}|\Box[ul]
		}
	\end{gathered}
\end{equation}
where $ \delta_0$ and $\delta_1$ are the diagonal morphisms. Put
$\mathcal K_{X/Y/S}:=\delta^\Delta\mathcal K_{X/S}\simeq \delta_1^\ast \delta^\Delta\delta_{0\ast}\mathcal K_{X/{S}}$. 
We have the following distinguished triangle, see \cite[(4.6)]{YZ22}:
\begin{equation}\label{eq:distriK}
\mathcal K_{X/{Y}}\to \mathcal K_{X/S}\to \mathcal K_{X/Y/S}\xrightarrow{+1}.
\end{equation}
Let $\mathcal F\in D_{\rm ctf}(X,\Lambda)$ such that 
$X\setminus Z\to Y$ is universally locally acyclic relatively 
to $\mathcal F|_{X\setminus Z}$ and that $h:X\to S$  is universally locally acyclic  
relatively to $\mathcal F$.
We put
\begin{equation}\label{eq:display-421}
\mathcal H_S=R\mathcal Hom_{X\times_SX}({\rm pr}_2^\ast \mathcal F,{\rm pr}_1^!\mathcal F),\qquad \mathcal T_S=\mathcal F\boxtimes^L_S D_{X/S}(\mathcal F).
\end{equation}
The relative cohomological characteristic class $C_{X/S}(\mathcal F)$ is
the composition (cf. \cite[(3.49)]{YZ22})
\begin{equation}\label{eq:defCharClassInfinity}
\Lambda\xrightarrow{\rm id}R\mathcal Hom(\mathcal F,\mathcal F)\xrightarrow[\simeq]{\eqref{eq:keyisoCshrinkRHOM}}\delta_0^! \mathcal H_S\xleftarrow[\simeq]{\eqref{eq:Kunethiso}}\delta_0^!\mathcal T_S
	\xrightarrow{}\delta_0^\ast \mathcal T_S \xrightarrow{\rm ev} \mathcal K_{X/S}.
\end{equation}
By the assumption on $\mathcal F$, $\delta_1^\ast \delta^\Delta\mathcal T_S$ is supported on $Z$ by \cite[\S4.3]{YZ22}.
The non-acyclicity class $\widetilde{C}^Z_{X/Y/S}(\mathcal F)$ is the composition (cf. \cite[Definition~4.2]{YZ22})
\begin{equation}\label{eq:priciseDefCZXYS}
\Lambda\to \delta^!_0 \mathcal H_S\xleftarrow{\simeq}\delta_0^!\mathcal T_S
	\simeq \delta_1^!i^!\mathcal T_S\xrightarrow{}\delta^\ast_1 i^!\mathcal T_S\to \delta_1^\ast \delta^\Delta\mathcal T_S\xleftarrow{\simeq} \tau_\ast\tau^! \delta_1^\ast \delta^\Delta\mathcal T_S\to \tau_\ast\tau^! \mathcal K_{X/Y/S}.
\end{equation}
If  the following condition holds:
\begin{equation}\label{eq:display-437}
H^0(Z,\mathcal K_{Z/{Y}})=0~{\rm and}~H^1(Z,\mathcal K_{Z/{Y}})=0
\end{equation}
then the map
\begin{equation}\label{eq:deltaf}
H^0_Z(X,\mathcal K_{X/{S}})
\xrightarrow{\eqref{eq:distriK}} H^0_Z(X,\mathcal K_{X/Y/{S}})
\end{equation}
is an isomorphism. In this case, the class 
$\widetilde{C}_{X/Y/S}^{Z}(\mathcal F)\in H_Z^0(X,\mathcal K_{X/Y/{S}})$
defines an element of $H^0_Z(X,\mathcal K_{X/{S}})$, which is  denoted by
$C_{X/Y/S}^{Z}(\mathcal F)$.
Now we summarize the functorial properties for the non-acyclicity classes.
\begin{proposition}[{\cite[Theorem~1.5, Proposition~1.6,
  Theorem~1.7 Theorem~1.8 and Proposition~4.7]{YZ22}}]
\label{NAprop}
	Let us denote the diagram \eqref{xyd:prop:Gysim:fgh2} by $\Delta=\Delta^Z_{X/Y/S}$ 
	and  $\widetilde{C}_{X/Y/S}^{Z}(\mathcal F)$, $C_{X/Y/S}^{Z}(\mathcal F)$ by $ \widetilde{C}_{\Delta}(\mathcal F)$, $ C_{\Delta}(\mathcal F)$, respectively. 
	Let $\mathcal F\in D_{\rm ctf}(X,\Lambda)$. 
	Assume that $Y\to S$ is smooth of pure relative dimension $r$, $X\setminus Z\to Y$ is universally locally acyclic relatively to $\mathcal F|_{X\setminus Z}$ and that $X\to S$ is universally locally acyclic relatively to $\mathcal F$.
	\begin{enumerate}
		\item $($Fibration formula$)$ 	If $H^0(Z,\mathcal K_{Z/{Y}})=H^1(Z,\mathcal K_{Z/{Y}})=0$,
		then we have
		\begin{equation}\label{eq:fibCNA}
C_{X/S}(\mathcal F)= c_{r}(f^*\Omega_{Y/S}^{1,\vee})\cap C_{X/Y}(\mathcal F)+
			C_{\Delta}(\mathcal F) \quad {\rm in }\quad H^{0}(X,\mathcal K_{X/S}).
\end{equation}
		\item $($Pullback$)$
		Let $b: S'\to S$ be a morphism of Noetherian schemes. Let $\Delta'=\Delta_{X'/Y'/S'}^{Z'}$ be the base change of $\Delta=\Delta^Z_{X/Y/S}$ by $b: S'\to S$. Let $b_X:X'=X\times_SS'\to X$  be the base change of $b$ by $X\to S$. Then
		we have
		\begin{equation}\label{eq:pbCNA}
b_{X}^\ast \widetilde{C}_{\Delta}(\mathcal F)=\widetilde{C}_{\Delta'}(b_X^\ast\mathcal F)\quad{\rm in}\quad H^0_{Z'}(X',\mathcal K_{X'/Y'/S'}),
\end{equation}
		where $b_{X}^\ast:H^0_{Z}(X,\mathcal K_{X/Y/S})\to H^0_{Z'}(X',\mathcal K_{X'/Y'/S'})$ is the induced pullback morphism.
		\item $($Proper pushforward$)$
		Consider a  diagram $\Delta'=\Delta^{Z'}_{X'/Y/S}$. Let $s:X\to X'$ be a proper morphism over $Y$ such that $Z\subseteq s^{-1}(Z')$.
		Then
		we have
		\begin{equation}\label{eq:pfCNA}
s_\ast (\widetilde{C}_{\Delta}(\mathcal F))=\widetilde{C}_{\Delta'}(Rs_\ast\mathcal F)\quad{\rm in}\quad H^0_{Z'}(X',\mathcal K_{X'/Y/S}),
\end{equation}
		where $s_\ast:H^0_{Z}(X,\mathcal K_{X/Y/S})\to H^0_{Z'}(X',\mathcal K_{X'/Y/S})$ is the induced  pushforward morphism.
		\item $($Cohomological Milnor formula$)$
		Assume $S=\Spec k$ for a perfect field $k$ of characteristic $p>0$, 
        $Y$ is a smooth curve over $k$
        and $\Lambda$ is a finite local ring such that the characteristic of the residue field is invertible in $k$. If $Z=\{x\}$, then we have
		\begin{equation}\label{eq:CMF}
C_{\Delta}(\mathcal F)=-\operatorname{dimtot} R\Phi_{\bar{x}}(\mathcal F,f)\quad {\rm in}\quad  \Lambda=H^0_{x}(X,\mathcal K_{X/k}),
\end{equation}
			where $R\Phi(\mathcal F,f)$ is  the complex of vanishing cycles and $\operatorname{dimtot}={\rm dim}+{\rm Sw}$ is the total dimension. 
		\item $($Cohomological conductor formula$)$ Assume $S=\Spec k$ for a perfect field $k$ of characteristic $p>0$ and $\Lambda$ is a finite local ring such that the characteristic of the residue field is invertible in $k$. If $Y$ is a smooth connected curve over $k$, $f$ is proper and  $Z=f^{-1}(y)$ for a closed point $y\in |Y|$, then we have 
			\begin{equation}\label{eq:CCF}
f_\ast \widetilde{C}_{\Delta}(\mathcal F)=-{ a}_y(R f_\ast\mathcal F)\quad {\rm in}\quad  \Lambda=H^0_{y}(Y,\mathcal K_{Y/k}),
\end{equation}
			where $a_y(\mathcal G)={\rm rank}\mathcal G|_{\bar\eta}-{\rm rank}\mathcal G_{\bar y}+{\rm Sw}_y\mathcal G$ is the Artin conductor of an object $\mathcal G\in D_{\rm ctf}(Y,\Lambda)$ at $y$ and $\eta$ is the generic point of $Y$.
        \item $($Specialization$)$
Assume that $S$ is a henselian trait with generic point $\eta$,
and let $s$ be a geometric point above its closed point.
Put $\Delta_t=\Delta^{Z_t}_{X_t/Y_t/t}$ and
$\mathcal F_t=\mathcal F|_{X_t}$ for $t\in\{\eta,s\}$.
For this item, the hypothesis on $f$ may be weakened to universal
local acyclicity of $f_t$ relatively to $\mathcal F_t$ outside $Z_t$.
Then we have
\begin{equation}\label{eq:na-specialization}
  \operatorname{sp}^{\mathrm{\acute et}}
  \bigl(\widetilde C_{\Delta_\eta}(\mathcal F_\eta)\bigr)
  =
  \widetilde C_{\Delta_s}(\mathcal F_s)
  \quad\text{in } H^0_{Z_s}(X_s,\mathcal K_{X_s/Y_s/s}),
\end{equation}
where 
\begin{equation}\label{eq:auto-002}
\operatorname{sp}^{\mathrm{\acute et}}:
H^0_{Z_\eta}(X_\eta,\mathcal K_{X_\eta/Y_\eta/\eta})
\longrightarrow
H^0_{Z_s}(X_s,\mathcal K_{X_s/Y_s/s})
\end{equation}
is the specialization map.
	\end{enumerate}
\end{proposition}

\subsection{}Let $X$ be a smooth connected curve over $k$. 
Let $\mathcal F\in D_{\rm ctf}(X,\Lambda)$ and $Z\subseteq X$ be a finite set of closed points such that the cohomology sheaves of $\mathcal F|_{X\setminus Z}$ are locally constant. By the cohomological Milnor formula  \eqref{eq:CMF}, we have the following (motivic) expression for the Artin conductor of $\mathcal F$ at $x\in Z$
\begin{equation}\label{eq:CAC}
a_x(\mathcal F)=\operatorname{dimtot}R\Phi_{\bar x}(\mathcal F,{\rm id})=-C^{\{x\}}_{U/U/k}(\mathcal F|_{U}),
\end{equation}
where $U$ is any open subscheme of $X$ such that $U\cap Z=\{x\}$. By \eqref{eq:fibCNA},
we get the following cohomological Grothendieck-Ogg-Shafarevich formula (cf. \cite[Corollary~6.4]{YZ22}):
\begin{equation}\label{eq:CGOS}
C_{X/k}(\mathcal F)={\rm rank}\mathcal F\cdot c_1(\Omega_{X/k}^{1,\vee})-\sum_{x\in Z}{a}_x(\mathcal F)\cdot[x]\quad {\rm in}\quad H^0(X,\mathcal K_{X/k}).
\end{equation}

\subsection{} Now we start to construct and prove a geometric counterpart of Proposition \ref{NAprop}.
Let $k$ be a perfect field of characteristic $p>0$ and $\Lambda$ be a finite local ring whose residue field is of characteristic $\ell\neq p$. Let ${\rm Sm}_k$ be the category of smooth schemes over $k$.
Let $S$ be a smooth connected scheme of dimension $s$ over $k$. Let $f:X\to S$ be a morphism in  ${\rm Sm}_k$. Let $\mathcal F\in D_{\rm ctf}(X,\Lambda)$ such that $f$ is $SS(\mathcal F)$-transversal. Consider the following morphisms
\begin{equation}\label{eq:display-550}
X\xrightarrow{0}T^{\ast} S\times_S X\xrightarrow{df}T^{\ast} X,
\end{equation}
where $0$ stands for the zero section. By assumption, $df^{-1}(SS(\mathcal F))$ is contained in $0(X)$.
We define the relative characteristic class of $\mathcal F$ to be the following $s$-cycle class on $X$:
\begin{equation}\label{eq:defRCC}
cc_{X/S}(\mathcal F):= (-1)^s\cdot (df)^!(CC(\mathcal F))\quad{\rm in}\quad {\rm CH}_s(X),
\end{equation}
where $(df)^!$ is the refined Gysin pullback.
When $f$ is a smooth morphism, then we have a Cartesian diagram
\begin{equation}\label{eq:display-561}
\begin{gathered}
\xymatrix{
T^{\ast} S\times_SX\ar[r]^-{df}\ar[d]\ar@{}[rd]|{\Box}&T^{\ast} X\ar[d]\\
X\ar[r]_-{0_{X/S}}&T^{\ast}(X/S).
}
\end{gathered}
\end{equation}
In this case, we have $cc_{X/S}(\mathcal F)=(-1)^s\cdot 0_{X/S}^!(CC(\mathcal F))$ (cf. \cite[Definition 2.11]{YZ18}). If $f$ is a smooth morphism of relative dimension $r$ and if $\mathcal F$ is locally constant, then we have 
\begin{equation}\label{eq:ccforconstant}
\begin{split}
cc_{X/S}(\mathcal F)
=(-1)^s\cdot 0_{X/S}^!((-1)^{{\rm dim}X}\cdot {\rm rank}\mathcal F\cdot [T^{\ast}_XX])
={\rm rank}\mathcal F\cdot c_r(\Omega_{X/S}^{1,\vee})\cap [X].
\end{split}
\end{equation}

We have the following conjecture, cf. \cite[Conjecture~3.8]{YZ18}:
\begin{conjecture}\label{conj:relativeSaitoConj}
Let $S$ be a smooth connected scheme of dimension $s$ over $k$. Let $f:X\to S$ be a morphism in  ${\rm Sm}_k$. Let $\mathcal F\in D_{\rm ctf}(X,\Lambda)$ such that $f$ is $SS(\mathcal F)$-transversal. Then we have
\begin{equation}\label{eq:conjC}
{\rm cl}(cc_{X/S}(\mathcal F))=C_{X/S}(\mathcal F)\quad{\rm in}\quad H^0(X,\mathcal K_{X/S}),
\end{equation}
where ${\rm cl}: {\rm CH}_s(X)\to H^0(X,\mathcal K_{X/S})$ is the cycle class map.
\end{conjecture}
When $S=\Spec k$, then
it is Saito's conjecture \cite[Conjecture 6.8.1]{Sai17a}, which is proved under quasi-projective assumption in \cite[Theorem~1.2]{YZ22}.
When $f: X\to S$ is a smooth morphism, then \eqref{eq:conjC} is true for a locally constant constructible (flat) sheaf $\mathcal F$ of $\Lambda$-modules. Indeed, this follows from \eqref{eq:ccforconstant}, \cite[Lemma~3.1]{YZ22} and \eqref{eq:fibCNA}.
\begin{remark}
    By \cite[Chapter~17]{Ful98}, ${\rm CH}_s(X)\cong {\rm CH}^0(X\to S)$ where the latter bivariant Chow group can be regarded as a relative version of ${\rm CH}_0(X)$. However, in general we cannot expect $cc_{X/S}(\mathcal{F})$ to lift canonically to a relative cycle.
\end{remark}
\subsection{}\label{subsec:defGNA}
Consider a commutative diagram
\begin{equation}\label{intro:deltadiag}
\begin{gathered}
\xymatrix{
Z\ar@{^(->}[r]^\tau&X\ar[rr]^f\ar[rd]_h&&Y\ar[ld]^g\\
&&S&
}
\end{gathered}
\end{equation}
where $X,Y,S$ belong to $\mathrm{Sm}_k$ and $\tau:Z\hookrightarrow X$
is a closed immersion. Assume that $X,Y,S$ are of pure dimension, that $S$ is connected
of dimension $s$, and that $g$ is smooth of relative dimension $r>0$.
Put $n=\dim X$.
Let $\mathcal F\in D_{\rm ctf}(X,\Lambda)$ such that $h$ is
$SS(\mathcal F)$-transversal and $f|_{X\setminus Z}$ is
$SS(\mathcal F|_{X\setminus Z})$-transversal.
Let
\begin{equation}\label{eq:projective-projection}
q:\mathbb{P}(T^*Y\times_YX)\longrightarrow X.
\end{equation}
be the projective bundle induced by $T^*Y\times_YX$ over $X$. The tautological subbundle and the differential give
\begin{equation}\label{eq:projectivized-differential-map}
\mathcal O(-1)\lhook\joinrel\longrightarrow
q^*(T^*Y\times_YX)\xrightarrow{q^*df}q^*T^*X.
\end{equation}
Under $\mathcal Hom(\mathcal O(-1),q^*T^*X)
=q^*T^*X\otimes\mathcal O(1)$, this homomorphism defines a section
of the vector bundle $q^*T^*X\otimes\mathcal O(1)$:
\begin{equation}\label{eq:overline-df}
\overline{df}:\mathbb{P}(T^*Y\times_YX)
\longrightarrow q^*T^*X\otimes\mathcal O(1).
\end{equation}
At a geometric point $(x,\lambda)$, with
$\lambda\subset T^*_{f(x)}Y$ a line, its value is the map
$df_x|_\lambda:\lambda\to T_x^*X$.
Thus $\overline{df}$ is defined even when $df_x(\lambda)=0$.
It is a section of a vector bundle.

For a conical closed subscheme $D\subset T^*X$, we define its twist.
On an open $V\subset\mathbb{P}(T^*Y\times_YX)$ with local basis $e$
of $\mathcal O(1)$, transport $D\times_XV$ by
\begin{equation}\label{eq:auto-003}
q^*T^*X|_V\xrightarrow{\sim}
(q^*T^*X\otimes\mathcal O(1))|_V,\qquad
\xi\longmapsto\xi\otimes e.
\end{equation}
Changing $e$ multiplies the vector coordinate by a unit.
Since the ideal of $D$ is homogeneous, these closed subschemes
agree on overlaps and define
$D_{\mathcal O(1)}\subset q^*T^*X\otimes\mathcal O(1)$.
Equivalently, if
$D=\operatorname{Spec}_X(\bigoplus_{a\ge0}\mathcal A_a)$
with its graded coordinate algebra, then
\begin{equation}\label{eq:cone-twist}
D_{\mathcal O(1)}
=\operatorname{Spec}_{\mathbb{P}(T^*Y\times_YX)}
\left(\bigoplus_{a\ge0}q^*\mathcal A_a\otimes\mathcal O(-a)\right).
\end{equation}
For $CC(\mathcal F)=\sum_a m_a[C_a]$, set
\begin{equation}\label{eq:CC-twist}
CC(\mathcal F)_{\mathcal O(1)}
:=\sum_a m_a[(C_a)_{\mathcal O(1)}].
\end{equation}
Each $(C_a)_{\mathcal O(1)}$ is integral of dimension $n+r+s-1$.
We apply the same twist separately to $SS(\mathcal F)$.

Notice that we have the following Cartesian diagram
\begin{equation}\label{eq:projective-gysin-square}
\begin{tikzcd}[ampersand replacement=\&,column sep=large,row sep=large]
\mathbb{P}((df)^{-1}SS(\mathcal F))
 \arrow[r]\arrow[d]\&SS(\mathcal F)_{\mathcal O(1)}\arrow[d]\\
\mathbb{P}(T^*Y\times_YX)
 \arrow[r,"\overline{df}"']\&q^*T^*X\otimes\mathcal O(1)
\end{tikzcd}
\end{equation}
which can be checked on open subschemes.
Moreover,
the section $\overline{df}$ is a regular immersion of codimension $n$, so
\begin{equation}\label{eq:projective-refined-class}
\overline{df}^{\,!}CC(\mathcal F)_{\mathcal O(1)}
\in{\rm CH}_{r+s-1}\bigl(\mathbb{P}((df)^{-1}SS(\mathcal F))\bigr).
\end{equation}
The hypotheses imply
\begin{equation}\label{eq:auto-004}
(df)^{-1}SS(\mathcal F)\subset0(X)\cup(T^*Y\times_YZ),\qquad
(df)^{-1}SS(\mathcal F)\cap(T^*S\times_SX)\subset0(X).
\end{equation}
Consequently, on $\mathbb{P}((df)^{-1}SS(\mathcal F))$, the tautological line bundle $\mathcal O(-1)$ is a subbundle of $q^* f^* T^*(Y/S)$,
by the short exact sequence of bundles
\begin{equation}\label{eq:auto-005}
0\rightarrow\ T^*S\times_SX\rightarrow T^*Y\times_Y X \rightarrow f^*T^*(Y/S)\rightarrow0.
\end{equation}
The reduced closed subscheme of the support $\mathbb{P}\left((df)^{-1}SS(\mathcal F)\right)$ maps properly to $Z_{\rm red}$ and hence to $Z$.
For this pushforward, denoted by $q_{Z*}$, we use the canonical
invariance of Chow groups under reduction. The refined intersection
retains its scheme structure and intersection multiplicities.

\begin{definition}\label{def:projectiveGNA}
We define the geometric non-acyclicity class of $\mathcal F$ by
\begin{equation}\label{eq:defGNA}
cc^Z_{X/Y/S}(\mathcal F)
:=(-1)^s q_{Z*}\left[
c_{r-1}\left(\frac{q^*f^*T^*(Y/S)}{\mathcal O(-1)}\right)
\cap\overline{df}^{\,!}CC(\mathcal F)_{\mathcal O(1)}
\right]\in{\rm CH}_s(Z).
\end{equation}
Here, the quotient bundle is taken on
$\mathbb{P}((df)^{-1}SS(\mathcal F))$.
\end{definition}

\begin{proposition}\label{prop:geoNA-elementary}
Consider the previous assumptions.
\begin{enumerate}
    \item If $f$ is $SS(\mathcal F)$-transversal on $X$, then
$cc^Z_{X/Y/S}(\mathcal F)=0$ for any $Z$;
    \item Let $Z\subset Z_1\subset X$ be closed subschemes, with inclusion $i:Z\hookrightarrow Z_1$. Then
    \begin{equation}\label{eq:auto-006}
i_*cc^Z_{X/Y/S}(\mathcal F)=cc^{Z_1}_{X/Y/S}(\mathcal F);
\end{equation}
    \item Let 
    \begin{equation}\label{eq:auto-007}
\mathcal F' \rightarrow \mathcal F\rightarrow \mathcal F'' \xrightarrow{+1}
\end{equation}
be a distinguished triangle in $\Dctf(X,\Lambda)$ 
with $\mathcal F'$, $\mathcal F''$ satisfy the transversality condition in the beginning of section \ref{subsec:defGNA}.
Then we have
    \begin{equation}\label{eq:auto-008}
        cc^Z_{X/Y/S}(\mathcal F)=
        cc^Z_{X/Y/S}(\mathcal F')+cc^Z_{X/Y/S}(\mathcal F'').
    \end{equation}
\end{enumerate}
\end{proposition}
\begin{proof}
The parts (1) and (2) are direct from definition. For (3), we use \cite[Lemma~5.13]{Sai17a}.
\end{proof}

\subsection{}

For a vector bundle $V$ of rank $d$ on $X$, put
\begin{equation}\label{eq:auto-009}
c_t(V)
=
t^d+c_1(V)t^{d-1}+\cdots+c_d(V)
\in \operatorname{CH}^*(X)[t].
\end{equation}
We regard $t$ as a formal variable of codimension one. Recall that
\begin{equation}\label{eq:auto-010}
\frac{1}{c_t(V)}
=
t^{-d}\sum_{a\geq 0}s_a(V)t^{-a},
\end{equation}
where $s(V)=c(V)^{-1}$ is the total Segre class.

Let $n=\dim X$, and write
\begin{equation}\label{eq:auto-011}
0_X:X\longrightarrow T^*X,
\qquad
q_X:\mathbb{P}(T^*X)\longrightarrow X
\end{equation}
for the zero section and the projection. Put
$\xi_X=c_1(\mathcal O_{\mathbb{P}(T^*X)}(1))$.
Write
\begin{equation}\label{eq:auto-012}
CC(\mathcal F)=0_{X*}\Gamma+\sum_a m_a[C_a],
\end{equation}
where $\Gamma$ is an $n$-cycle on $X$ and no $C_a$ is contained in the
zero section. Since each $C_a$ is conical,
$\mathbb{P}(C_a)$ is a closed integral subscheme of
$\mathbb{P}(T^*X)$ of dimension $n-1$.
By the projective bundle formula, there are classes
\begin{equation}\label{eq:beta-explicit}
\beta_j
={q_X}_*\!\left[
c_{n-1-j}\!\left(
\frac{q_X^*T^*X}{\mathcal O(-1)}
\right)
\cap\sum_a m_a[\mathbb P(C_a)]
\right],
\qquad 0\leq j\leq n-1
\end{equation}
such that
\begin{equation}\label{eq:auto-014}
\sum_a m_a[\mathbb{P}(C_a)]
=
\sum_{j=0}^{n-1}
\xi_X^j\cap q_X^*\beta_j
\quad\text{in }
\operatorname{CH}_{n-1}(\mathbb{P}(T^*X)).
\end{equation}
Define $b_j\in\operatorname{CH}_j(X)$ by
\begin{equation}
\label{eq:projective-characteristic-polynomial}
\sum_{j=0}^n b_jt^j
=
c_t(T^*X)\cap\Gamma
+
\sum_{j=0}^{n-1}\beta_jt^j,
\end{equation}
where
\begin{equation}\label{eq:auto-015}
c_t(T^*X)\cap\Gamma
=
\sum_{i=0}^n
\bigl(c_i(T^*X)\cap\Gamma\bigr)t^{n-i}.
\end{equation}
In particular, $b_n=\Gamma$. The classes $b_j$ are canonically
determined by $CC(\mathcal F)$. Indeed, by the projective bundle relation
\begin{equation}\label{eq:auto-016}
\xi_X^n+c_1(T^*X)\xi_X^{n-1}
+\cdots+c_n(T^*X)=0
\end{equation}
we have
\begin{equation}
\label{eq:pcccal}
    \sum_{j=0}^n
\xi_X^j\cap q_X^*b_j
=
\sum_a m_a[\mathbb{P}(C_a)].
\end{equation}
Put $m=\dim Y=r+s$. We define
\begin{equation}
\label{eq:relative-characteristic-class}
\overline{cc}_{X/Y}(\mathcal F)
:=
(-1)^m
\sum_{j\geq m}
s_{j-m}(T^*Y\times_YX)\cap b_j
\in\operatorname{CH}_m(X).
\end{equation}
Equivalently, $\overline{cc}_{X/Y}(\mathcal F)$ is the constant coefficient of the Laurent
expansion at $t=\infty$ of
\begin{equation}\label{eq:auto-017}
(-1)^m\cdot
{c_t(T^*Y\times_YX)}^{-1}
\cdot
\sum_{j=0}^n b_jt^j.
\end{equation}

\begin{proposition}[fibration formula]
Consider the previous assumptions, then
the restriction of $\overline{cc}_{X/Y}(\mathcal F)$ to $X\setminus Z$ is
$cc_{(X\setminus Z)/Y}(\mathcal F|_{X\setminus Z})$.
Moreover,
\begin{equation}
\label{eq:relative-characteristic-fibration}
cc_{X/S}(\mathcal F)
=
c_r\bigl((f^*\Omega^1_{Y/S})^\vee\bigr)
\cap \overline{cc}_{X/Y}(\mathcal F)
+
\tau_*cc^Z_{X/Y/S}(\mathcal F)
\quad\text{in }\operatorname{CH}_s(X).
\end{equation}
In particular, if $\dim Z<r+s$, the restriction
\begin{equation}\label{eq:auto-018}
\operatorname{CH}_{r+s}(X)
\longrightarrow
\operatorname{CH}_{r+s}(X\setminus Z)
\end{equation}
is an isomorphism, and $\overline{cc}_{X/Y}(\mathcal F)$ is the unique extension of
$cc_{(X\setminus Z)/Y}(\mathcal F|_{X\setminus Z})$. If $Z=\varnothing$, then
\begin{equation}
\label{eq:relative-characteristic-fibration-transversal}
cc_{X/S}(\mathcal F)
=
c_r\bigl((f^*\Omega^1_{Y/S})^\vee\bigr)
\cap cc_{X/Y}(\mathcal F).
\end{equation}
\end{proposition}

\begin{proof}
Since $h:X\to S$ is $SS(\mathcal F)$-transversal, the inverse image of $SS(\mathcal F)$
under
\begin{equation}\label{eq:auto-019}
dh:T^*S\times_SX\longrightarrow T^*X
\end{equation}
is contained in the zero section,
hence every projectivized
nonzero component of $CC(\mathcal F)$ is disjoint from
$\mathbb{P}(T^*S\times_S X)$. Hence \eqref{eq:pcccal} restricts to
\begin{equation}\label{eq:auto-020}
\sum_j \xi^j\cap q^*b_j=0
\end{equation}
on $\mathbb{P}(T^*S\times_S X)$.
By the projective bundle formula, the polynomial
$\sum_j b_jt^j$ is divisible by
$c_t(T^*S\times_S X)$.
Write
\begin{equation}\label{eq:auto-021}
\sum_{j=0}^n b_jt^j
=
c_t(T^*S\times_SX)\sum_{j\geq0}a_jt^j,
\qquad
a_j\in\operatorname{CH}_{s+j}(X).
\end{equation}
The constant coefficient is the refined intersection with
$dh$:
\begin{equation}\label{eq:auto-022}
a_0=(dh)^!CC(\mathcal F),
\qquad
cc_{X/S}(\mathcal F)=(-1)^sa_0.
\end{equation}
Since $g:Y\to S$ is smooth of relative dimension $r$, there is an exact
sequence of vector bundles on $X$
\begin{equation}\label{eq:auto-023}
0\longrightarrow T^*S\times_SX
\longrightarrow T^*Y\times_YX
\longrightarrow f^*T^*(Y/S)
\longrightarrow0.
\end{equation}
Hence
\begin{equation}\label{eq:auto-024}
c_t(T^*Y\times_YX)
=
c_t(T^*S\times_SX)c_t(f^*T^*(Y/S)).
\end{equation}
It follows from \eqref{eq:relative-characteristic-class} that
\begin{equation}
\label{eq:relative-characteristic-aj}
\overline{cc}_{X/Y}(\mathcal F)
=
(-1)^{r+s}
\sum_{j\geq r}
s_{j-r}(f^*T^*(Y/S))\cap a_j.
\end{equation}
The projection
$T^*Y\times_YX\longrightarrow f^*T^*(Y/S)$
has kernel
$T^*S\times_SX$, and
\begin{equation}\label{eq:auto-025}
(df)^{-1}SS(\mathcal F)\cap(T^*S\times_SX)\subset0(X)
\end{equation}
by the $SS(\mathcal F)$-transversality of $h$,
hence it
is finite on $(df)^{-1}SS(\mathcal F)$. 
Let
\begin{equation}\label{eq:auto-026}
q_{S}:\mathbb{P}(f^*T^*(Y/S))\longrightarrow X
\end{equation}
be the projective bundle projection, and put
$\xi=c_1(\mathcal O_{\mathbb{P}(f^*T^*(Y/S))}(1))$.
The preceding finite projection induces a finite morphism
\begin{equation}\label{eq:auto-027}
\rho:
\mathbb{P}((df)^{-1}SS(\mathcal F))
\longrightarrow
\mathbb{P}(f^*T^*(Y/S))
\end{equation}
which identifies the tautological line bundles; hence we may slightly abuse the tautological bundle notation here. 
By the construction of the classes $b_j$ we have
\begin{equation}\label{eq:auto-028}
\rho_*
\bigl(\overline{df}^!CC(\mathcal F)_{\mathcal O(1)}\bigr)
=
\sum_{j\geq0}\xi^j\cap q_{S}^*a_j.
\end{equation}
On $\mathbb{P}(f^*T^*(Y/S))$ we have the canonical quotient bundle
$q_{S}^*f^*T^*(Y/S)/\mathcal O(-1)$,
projective bundle formula gives
\begin{equation}
\label{eq:projective-bundle-computation}
(q_{S})_*
\left[
c_{r-1}\left(
q_{S}^*f^*T^*(Y/S)-\mathcal O(-1)
\right)
\cap\xi^j
\right]
=
\begin{cases}
1, & j=0,\\
0, & 0<j<r,\\
-c_r(f^*T^*(Y/S))
\,s_{j-r}(f^*T^*(Y/S)), & j\geq r.
\end{cases}
\end{equation}
Applying \eqref{eq:projective-bundle-computation} and the definition of
$cc^Z_{X/Y/S}(\mathcal F)$, we obtain
\begin{equation}\label{eq:auto-029}
\tau_*cc^Z_{X/Y/S}(\mathcal F)
=
(-1)^s
\left[
a_0
-
c_r(f^*T^*(Y/S))
\cap
\sum_{j\geq r}
s_{j-r}(f^*T^*(Y/S))\cap a_j
\right].
\end{equation}
By
\eqref{eq:relative-characteristic-aj},
the second term on the right is
$c_r\bigl((f^*T^*(Y/S))^\vee\bigr)
\cap \overline{cc}_{X/Y}(\mathcal F)$.
Together with $cc_{X/S}(\mathcal F)=(-1)^sa_0$, this proves
\eqref{eq:relative-characteristic-fibration}.

It remains to prove the restriction assertion. Put $U=X\setminus Z$.
On $U$, the morphism
$f|_U:U\longrightarrow Y$
is $SS(\mathcal F|_U)$-transversal. Restricting
\eqref{eq:projective-characteristic-polynomial} to $U$, the
zero-section excess formula gives
\begin{equation}\label{eq:auto-030}
\sum_j(b_j|_U)t^j
=
c_t(T^*Y\times_YU)\sum_{j\geq0}d_jt^j,
\qquad
d_0=(df)^!CC(\mathcal F|_U).
\end{equation}
Consequently,
\begin{equation}\label{eq:auto-031}
\begin{aligned}
\overline{cc}_{X/Y}(\mathcal F)|_U
&=
(-1)^md_0
=
cc_{U/Y}(\mathcal F|_U).
\end{aligned}
\end{equation}
This proves the restriction statement.

Finally, if $\dim Z<r+s$, then
$\operatorname{CH}_{r+s}(Z)=0$.
Hence we have an isomorphism
\begin{equation}\label{eq:auto-032}
\operatorname{CH}_{r+s}(X)
\xrightarrow{\sim}
\operatorname{CH}_{r+s}(X\setminus Z),
\end{equation}
by the localization sequence,
and $\overline{cc}_{X/Y}(\mathcal F)$ is the unique extension of
$cc_{(X\setminus Z)/Y}(\mathcal F|_{X\setminus Z})$.
Equation
\eqref{eq:relative-characteristic-fibration-transversal}
is the special case $Z=\varnothing$.
\end{proof}

\begin{remark}
    In fact, the construction of $\overline{cc}_{X/Y}(\mathcal F)$ can be regarded as a generalized definition of the characteristic class, weakening the transversality condition.
\end{remark}

\begin{proposition}\label{prop:geoNA-affine-line}
Let $S=\Spec k$, $Y=\mathbb A^1_k$ with coordinate $t$, and let
$Z$ contain the characteristic locus of $f$.
Writing $df:X\longrightarrow X\times_{\A^1_k} T^*\A^1_k\longrightarrow T^*X$ for the section defined by $f^*dt$, one has
\begin{equation}\label{eq:geoNA-affine-line}
cc^Z_{X/\mathbb A^1_k/k}(\mathcal F)
=(df)^!CC(\mathcal F)\quad\text{in }{\rm CH}_0(Z).
\end{equation}
\end{proposition}
\begin{proof}
The section defined by $dt$ identifies
$\mathbb{P}(T^*\mathbb A^1_k\times_{\mathbb A^1_k}X)$ with $X$,
trivializes $\mathcal O(-1)$, and identifies $\overline{df}$ with $df$.
In \eqref{eq:defGNA}, $r=1$, $s=0$ and the Chern class involved is $c_0=1$.
\end{proof}

\begin{theorem}[Saito's isolated Milnor formula]\label{thm:SMF}
Assume $S=\Spec k$ and $Y$ is a smooth curve.
Let $Z$ be a finite set of closed points of $X$, and assume
$f$ is $SS(\mathcal F)$-transversal outside $Z$. Then
\begin{equation}\label{eq:SMF}
cc^Z_{X/Y/k}(\mathcal F)
=-\sum_{x\in Z}\operatorname{dimtot}R\Phi_{\bar x}(\mathcal F,f)\,[x]
\quad\text{in }{\rm CH}_0(Z).
\end{equation}
Here total dimension is computed after derived reduction of $\Lambda$
to its residue field, with alternating signs for complexes.
\end{theorem}
\begin{proof}
Near $f(x)$ choose an \'etale coordinate on $Y$, we may assume
$Y=\A^1_k$.
The consequence follows from Saito's Milnor formula \cite[Theorem 5.9]{Sai17a}.
\end{proof}

We expect the following Milnor type formula for non-isolated singular/characteristic points holds.
\begin{conjecture}\label{conj:MFforNI} 
Let $S$ be a smooth connected $k$-scheme of dimension $s$.
Consider the commutative diagram \eqref{intro:deltadiag}.
Let $\mathcal F\in D_{\rm ctf}(X,\Lambda)$ such that $X\setminus Z\to Y$ is $SS(\mathcal F|_{X\setminus Z})$-transversal and that $X\to S$ is $SS(\mathcal F)$-transversal.
Then we have an equality
	\begin{equation}\label{MFforNonIsolated}
\widetilde{C}^Z_{X/Y/S}(\mathcal F)=\widetilde{{\rm cl}}(cc^Z_{X/Y/S}(\mathcal F))\qquad{\rm in}\qquad H^0_Z(X,\mathcal K_{X/Y/S}),
\end{equation}
	where $\widetilde{{\rm cl}}$ is the composition
	${\rm CH}_s(Z)\xrightarrow{\rm cl} H^0_Z(X,\mathcal K_{X/S})\xrightarrow{\eqref{eq:distriK}} H^0_Z(X,\mathcal K_{X/Y/S})$.
\end{conjecture}
When $S=\Spec k$, $Y=\mathbb A_k^1$ and $Z=\{x\}$, then  Conjecture \ref{conj:MFforNI} follows from Saito's isolated Milnor formula \eqref{eq:SMF} and the cohomological Milnor formula \eqref{eq:CMF}.

\begin{remark}
Let $f:X\to Y$ be a separated morphism between smooth schemes over $k$, $X$ is connected.
Let $Z\subseteq X$ be a closed subset and $\mathcal F\in D_{\rm ctf}(X,\Lambda)$.
Assume that $f$ is universally locally acyclic relatively to $\mathcal F$ outside $Z$.	
Let $n={\rm dim}X$.
We expect that there is a $n$-cycle $CC^Z_{X/Y}(\mathcal F)$ supported on $T^{\ast} X\times_XZ$
such that 
	\begin{equation}\label{eq:display-704}
\widetilde{{\rm cl}}(0_X^!CC^Z_{X/Y}(\mathcal F))=\widetilde{C}^Z_{X/Y/k}(\mathcal F) \quad{\rm in}\quad H^0_Z(X,\mathcal K_{X/Y/k}).
\end{equation}	
If $Y$ is a smooth curve and $Z$ is a finite set of closed points of $X$, then 
	\begin{equation}\label{eq:display-708}
CC^Z_{X/Y}(\mathcal F)=-\sum_{x\in Z}\operatorname{dimtot}R\Phi_{\bar{x}}(\mathcal F,f)\cdot [T^{\ast}_xX].
\end{equation}
If $f={\rm id}$ and $Z$ is the smallest closed subset of $X$ such that $\mathcal F|_{X\setminus Z}$ is smooth, then
	\begin{equation}\label{eq:display-712}
CC^Z_{X/Y}(\mathcal F)=CC(\mathcal F)-{\rm rank}\mathcal F \cdot CC(\Lambda).
\end{equation}	
\end{remark}

\begin{proposition}
	Consider a Cartesian diagram in ${\rm Sm}_k$
	\begin{equation}\label{eq:pullbackFixDia}
\begin{gathered}
			\xymatrix{ X'\ar[d]_-{f'} \ar[r]^-{i} & X  \ar[d]^-f\ar@{}[dl]|\Box \\
				S'\ar[r]_-{\delta} &S.
			}
		\end{gathered}
\end{equation}
	We assume that $f$ and $f'$ are smooth morphisms, and $S$ and $S'$ are connected  of  dimension $s$ and $s'$ respectively.
	Let $\mathcal F\in D_{\rm ctf}(X,\Lambda)$.
	Assume that $X\to S$ is $SS(\mathcal F)$-transversal and $i$ is properly $SS(\mathcal{F})$-transversal.
	Then we have
	\begin{equation}\label{eq:display-736}
		 {i}^!cc_{X/S}(\mathcal{F})=cc_{X'/S'}(i^*\mathcal{F}) \quad{\rm in}\quad  {\rm CH}_{s'}(X'),
	\end{equation}
	where $i^!\colon {\rm CH}_s(X)\to {\rm CH}_{s'}(X')$ is the refined Gysin pullback.
\end{proposition}
\begin{proof}
	We consider the following diagram
	\begin{equation}\label{eq:display-746}
	\begin{tikzcd}[row sep=2em,column sep=1em]
		&T^{\ast}X'\arrow[dd]  && T^{\ast}X\times_{X}X' \arrow[dd] \arrow[ll, swap,"di"] \arrow[rr, "{\rm pr}"]&& T^{\ast}X\arrow[dd]\\
		T^{\ast}S'\times_{S'}X' \arrow[dd] \arrow[ur]&& T^{\ast}S\times_{S}X' \arrow[ur]  \arrow[ll, crossing over, swap, "d\delta" near start] \arrow[rr,crossing over] &&T^{\ast}S\times_SX\arrow[ur]& \\
		& T^{\ast}(X'/S') \arrow[rr, "\simeq" near start]  && T^{\ast}(X/S)\times_XX'  \arrow[rr] &&T^{\ast}(X/S) \\
		X'\arrow[ur,swap, "0_{X'/S'}"]   \arrow[rr, "="]&& X' \arrow[uu,<-,crossing over] \arrow[ur, swap, "0_{X/S} \times 1"]  \arrow[rr,"i"] && X \arrow[uu,<-,crossing over] \arrow[ur, swap,"0_{X/S}"]&
	\end{tikzcd}
	\end{equation}
	Note that the square containing the morphisms $di$ and $d\delta$ is Cartesian. In the following calculations,
		although $di$ and $d\delta$ are not proper, we can still apply $di_\ast$ and $d\delta_\ast$
	since $di$ is finite on the support of ${\rm pr}^{-1}(SS(\mathcal F))$ and $d\delta$ is finite on the zero section $X'$ of $T^{\ast} S\times_S X'$.
	We have
	\begin{equation}\label{eq:display-758}
\begin{split}
		cc_{X'/S'}(i^*\mathcal{F})&=(-1)^{s'}\cdot 0^!_{X'/S'}CC(i^*\mathcal F)\\
		&\overset{(a)}{=}(-1)^{s'}\cdot 0^!_{X'/S'}(di_*{\rm pr}^!CC(\mathcal F)\cdot (-1)^{-\text{dim}(X')+\text{dim}(X)})\\
		&=d\delta_*0_{X'/S'}^!{\rm pr}^!CC(\mathcal F)\cdot (-1)^{s'-{\rm dim}(X')+{\rm dim}(X)}\\
		&\overset{(b)}{=}0_{X'/S'}^!{\rm pr}^!CC(\mathcal F)\cdot (-1)^{s}=0_{X/S}^!{\rm pr}^!CC(\mathcal F)\cdot (-1)^{s}\\
		&=i^!(0_{X/S}^!CC(\mathcal F)\cdot (-1)^{s})=i^!cc_{X/S}(\mathcal F).
		\end{split}
\end{equation}
	where $(a)$ follows from \cite[Theorem 7.6]{Sai17a} and $(b)$ follows from the fact that $0^!_{X'/S'}{\rm pr}^!CC(\mathcal F)$ is supported on the zero section of $T^{\ast} S\times_S X'$.
\end{proof}

\subsection{}\label{subsec:pfcharclass}
Consider a commutative diagram in ${\rm Sm}_k$:
\begin{equation}\label{eq:display-773}
\begin{gathered}
		\xymatrix{
			X\ar[rr]^-f\ar[rd]_-h&&Y\ar[ld]^-g\\
			&S&
		}.
	\end{gathered}
\end{equation}
Let $\mathcal F\in D_{\rm ctf}(X,\Lambda)$ such that $h$ is $SS(\mathcal F)$-transversal.
Assume $f$ is proper on ${\rm supp}(\mathcal F)$.
By \cite[Lemma~3.8]{Sai17a}, $g$ is $f_\circ SS(\mathcal F)$-transversal.
By \cite[Lemma~2.2(ii)]{Bei16}, $SS(Rf_\ast\mathcal F)\subseteq f_\circ SS(\mathcal F)$.
Thus $g:Y\to S$ is also $SS(Rf_\ast\mathcal F)$-transversal and the class $cc_{Y/S}(Rf_\ast\mathcal F)$ is well-defined.
\begin{proposition}\label{prop:pfcharclass}
Consider the assumptions in section \ref{subsec:pfcharclass}.
Assume moreover that $Y$ is projective, $f: X\to Y$ is quasi-projective 
and ${\rm dim}f_\circ SS(\mathcal F)\leq {\rm dim}Y$. Then we have 
\begin{equation}\label{eq:display-790}
f_\ast cc_{X/S}(\mathcal F)=cc_{Y/S}(Rf_\ast\mathcal F)\quad{\rm in}\quad {\rm CH}_s(Y).
\end{equation}
\end{proposition}
\begin{proof}
Consider the following commutative diagram
\begin{equation}\label{eq:display-797}
\begin{gathered}
\xymatrix{
&X\ar[r]^-f\ar@{^(->}[d]^-{0_X}&Y\ar@{^(->}[d]^-{0_Y}\\
&T^{\ast} S\times_SX\ar[r]^-{{\rm id}\times f}\ar[d]^-{r}\ar[ld]_-{dh}&T^{\ast} S\times_SY\ar[d]^-{dg}\\
T^{\ast} X&T^{\ast} Y\times_Y X\ar[r]^-{{\rm id}\times f}\ar[l]^-{df}&T^{\ast} Y
}
\end{gathered}
\end{equation}
Then we have
\begin{equation}\label{eq:display-807}
\begin{split}
	f_\ast cc_{X/S}(\mathcal F)&\overset{\eqref{eq:defRCC}}{=}(-1)^s\cdot f_\ast dh^{!}(CC(\mathcal F))=(-1)^s\cdot ({\rm id}\times f)_\ast  r^! df^{!}(CC(\mathcal F))\\
	&=(-1)^s\cdot dg^!({\rm id}\times f)_\ast df^!(CC(\mathcal F))\\
	&\overset{(a)}{=} (-1)^s\cdot dg^! CC(Rf_\ast\mathcal F)\overset{\eqref{eq:defRCC}}{=}cc_{Y/S}(Rf_\ast\mathcal F).
\end{split}
\end{equation}
where (a) follows from {\cite[Theorem 2.2.5]{Sai21}}.
\end{proof}

\begin{lemma}\label{lem:projective-relative-expression}
Consider the assumptions in section \ref{subsec:defGNA}. 
On an open neighbourhood of ${\rm supp}(\mathcal F)$, the morphism
$h:X\to S$ is smooth. Let
\begin{equation}\label{eq:auto-033}
\pi_{X/S}:T^*X\longrightarrow T^*(X/S),\qquad
 q_{S}:\mathbb{P}(f^*T^*(Y/S))\longrightarrow X
\end{equation}
be the natural projections on this neighbourhood. The restriction of
$\pi_{X/S}$ to $SS(\mathcal F)$ is finite. The relative differential
$\mathcal O(-1)\to q_{S}^*T^*(X/S)$ defines a section
\begin{equation}\label{eq:auto-034}
\overline{d_{/S}f}:\mathbb{P}(f^*T^*(Y/S))
 \longrightarrow q_{S}^*T^*(X/S)\otimes\mathcal O(1),
\end{equation}
and we have
\begin{equation}\label{eq:projective-relative-expression}
cc^Z_{X/Y/S}(\mathcal F)=(-1)^s(q_{S})_*\bigg[
 c_{r-1}\!\left(
 \frac{q_{S}^*f^*T^*(Y/S)}{\mathcal O(-1)}\right)
 \cap\overline{d_{/S}f}^{\,!}
 \bigl((\pi_{X/S})_*CC(\mathcal F)\bigr)_{\mathcal O(1)}\bigg].
\end{equation}
Here $(\pi_{X/S})_*$ is taken on the characteristic-cycle support, and
$(q_{S})_*$ is taken on the reduced intersection support over $Z$.
\end{lemma}
\begin{proof}
The morphism $h$ is smooth on an open neighbourhood of
${\rm supp}(\mathcal F)$ by \cite[Lemmas~3.6(6) and~4.2(3)]{Sai17a}, and we have
\begin{equation}\label{eq:auto-035}
0\longrightarrow T^*S\times_SX
\xrightarrow{dh}T^*X
\xrightarrow{\pi_{X/S}}T^*(X/S)
\longrightarrow0.
\end{equation}
on the open neighbourhood.
Moreover,
\begin{equation}\label{eq:auto-036}
SS(\mathcal F)\cap\ker(\pi_{X/S})\subset0(X),
\end{equation}
so $\pi_{X/S}|_{SS(\mathcal F)}$ is finite by 
\cite[proof of Lemma~1.2(ii)]{Bei16}.

The following square is Cartesian, since its vertical kernels are
identified by $dh$:
\begin{equation}\label{eq:auto-037}
\begin{tikzcd}[column sep=large]
T^*Y\times_YX\arrow[r,"df"]\arrow[d]
&T^*X\arrow[d,"\pi_{X/S}"]\\
f^*T^*(Y/S)\arrow[r,"d_{/S}f"']
&T^*(X/S).
\end{tikzcd}
\end{equation}
The tautological line bundle therefore gives
\begin{equation}\label{eq:auto-038}
\mathcal O(-1)
\longrightarrow q_{S}^*f^*T^*(Y/S)
\longrightarrow q_{S}^*T^*(X/S),
\end{equation}
hence we have a section $\overline{d_{/S}f}$ as in the statement.
Since $\pi_{X/S}$ is finite on $SS(\mathcal F)$, Gysin pullback commutes  with 
proper pushforward
\cite[Theorem~6.2]{Ful98}, hence we have
\begin{equation}\label{eq:auto-039}
\rho_*\,df^!CC(\mathcal F)
=
(d_{/S}f)^!(\pi_{X/S})_*CC(\mathcal F).
\end{equation}
By transversality,
\begin{equation}\label{eq:auto-040}
(df)^{-1}SS(\mathcal F)\cap(T^*S\times_SX)\subset0(X),
\end{equation}
hence
$T^*Y\times_YX\longrightarrow f^*T^*(Y/S)$
induces a finite map
\begin{equation}\label{eq:auto-041}
\bar\rho:
\mathbb{P}((df)^{-1}SS(\mathcal F))
\longrightarrow
\mathbb{P}\bigl((d_{/S}f)^{-1}
(\pi_{X/S}SS(\mathcal F))\bigr),
\end{equation}
and $\bar\rho^*\mathcal O(-1)=\mathcal O(-1)$.
By the compatibility with proper pushforward of Gysin pullback
we have
\begin{equation}\label{eq:auto-042}
\bar\rho_*
\bigl(\overline{df}^{\,!}
CC(\mathcal F)_{\mathcal O(1)}\bigr)
=
\overline{d_{/S}f}^{\,!}
\bigl((\pi_{X/S})_*CC(\mathcal F)\bigr)_{\mathcal O(1)},
\end{equation}
hence
\begin{equation}\label{eq:auto-043}
\bar\rho^*
\left(
\frac{q_{S}^*f^* T^*(Y/S)}{\mathcal O(-1)}
\right)
=
\frac{q^*f^*T^*(Y/S)}{\mathcal O(-1)},
\end{equation}
and the projection formula and \eqref{eq:defGNA} therefore implies
\begin{equation}\label{eq:auto-044}
cc^Z_{X/Y/S}(\mathcal F)=(-1)^s(q_{S})_*\bigg[
c_{r-1}\left(
\frac{q_{S}^*f^*T^*(Y/S)}{\mathcal O(-1)}
\right)
\cap\overline{d_{/S}f}^{\,!}
\bigl((\pi_{X/S})_*CC(\mathcal F)\bigr)_{\mathcal O(1)}
\bigg].
\end{equation}
which is \eqref{eq:projective-relative-expression}. 
\end{proof}

\begin{proposition}\label{prop:pbGNA}
Consider a commutative diagram in ${\rm Sm}_k$
\begin{equation}\label{eq:pullbackFixDiaXYS}
\begin{tikzcd}[column sep=large]
 X'\arrow[r,"i_X"]\arrow[d,"f'"']&X\arrow[d,"f"]\\
 Y'\arrow[r,"i_Y"]\arrow[d,"g'"']&Y\arrow[d,"g"]\\
 S'\arrow[r,"\delta"']&S,
\end{tikzcd}
\end{equation}
whose squares are Cartesian. Put $Z'=Z\times_XX'$, and write
$s'=\dim S'$. Assume that $X',Y',S'$ are smooth and pure-dimensional,
that $S'$ is connected,
that $\dim X'-\dim X=s'-s$, and that $i_X$ is properly
$SS(\mathcal F)$-transversal. Then
\begin{equation}\label{eq:pbGNA}
 i_X^!cc^Z_{X/Y/S}(\mathcal F)
 =cc^{Z'}_{X'/Y'/S'}(i_X^*\mathcal F)
 \quad\text{in }{\rm CH}_{s'}(Z').
\end{equation}
\end{proposition}

\begin{proof}
Using the Cartesian diagram \eqref{eq:pullbackFixDiaXYS}, proper
$SS(\mathcal F)$-transversality gives
\begin{equation}\label{eq:auto-045}
SS(i_X^*\mathcal F)
\subset
di_X\bigl({\rm pr}^{-1}SS(\mathcal F)\bigr),
\qquad
{\rm pr}:T^*X\times_XX'\longrightarrow T^*X.
\end{equation}
Hence
\begin{equation}\label{eq:auto-046}
SS(i_X^*\mathcal F)\cap(T^*S'\times_{S'}X')
\subset0(X'),
\qquad
(df')^{-1}SS(i_X^*\mathcal F)
\subset
0(X')\cup(T^*Y'\times_{Y'}Z').
\end{equation}
Thus $g'\circ f'$ is $SS(i_X^*\mathcal F)$-transversal and
$f'|_{X'\setminus Z'}$ is
$SS((i_X^*\mathcal F)|_{X'\setminus Z'})$-transversal. 
Pullback formula of characteristic cycles
\cite[Theorem~7.6]{Sai17a} gives
\begin{equation}\label{eq:auto-047}
CC(i_X^*\mathcal F)
=
(-1)^{\dim X'-\dim X}
(di_X)_*{\rm pr}^!CC(\mathcal F),
\end{equation}
then by base change we have
\begin{equation}\label{eq:auto-048}
(\pi_{X'/S'})_*CC(i_X^*\mathcal F)
=
(-1)^{\dim X'-\dim X}
\,i_{X/S}^!(\pi_{X/S})_*CC(\mathcal F),
\end{equation}
where
\begin{equation}\label{eq:auto-049}
i_{X/S}:
T^*(X'/S')
\simeq T^*(X/S)\times_XX'
\longrightarrow T^*(X/S)
\end{equation}
is the natural projection.
Let
$i_{\mathbb{P}}:
\mathbb{P}(f'^*T^*(Y'/S'))
\longrightarrow
\mathbb{P}(f^*T^*(Y/S))$
be the induced morphism, then
\begin{equation}\label{eq:auto-050}
\begin{tikzcd}[column sep=large]
\mathbb{P}(f'^*T^*(Y'/S'))
\arrow[r,"i_{\mathbb{P}}"]\arrow[d,"q_{S'}"']
&
\mathbb{P}(f^*T^*(Y/S))
\arrow[d,"q_{S}"]\\
X'\arrow[r,"i_X"']&X
\end{tikzcd}
\end{equation}
is Cartesian, and
$
i_{\mathbb{P}}^*\mathcal O(-1)=\mathcal O(-1)$,
$i_{\mathbb{P}}^*q_{S}^*f^*T^*(Y/S)
=
q_{S'}^*f'^*T^*(Y'/S').
$
Thus
\begin{equation}\label{eq:auto-051}
i_{\mathbb{P}}^*
\left(
\frac{q_{S}^*f^*T^*(Y/S)}{\mathcal O(-1)}
\right)
=
\frac{q_{S'}^*f'^*T^*(Y'/S')}{\mathcal O(-1)}.
\end{equation}
The relative differentials commute with base change, giving the
following Cartesian diagram:
\begin{equation}\label{eq:auto-052}
\begin{tikzcd}[column sep=large,row sep=large]
\mathbb{P}(f'^*T^*(Y'/S'))
\arrow[r,"i_{\mathbb{P}}"]
\arrow[d,"\overline{d_{/S'}f'}"']
&
\mathbb{P}(f^*T^*(Y/S))
\arrow[d,"\overline{d_{/S}f}"]
\\
q_{S'}^*T^*(X'/S')\otimes\mathcal O(1)
\arrow[r]
&
q_S^*T^*(X/S)\otimes\mathcal O(1).
\end{tikzcd}
\end{equation}
Here the lower horizontal arrow is the natural projection under
the identification
$q_{S'}^*T^*(X'/S')\otimes\mathcal O(1)
\simeq
i_{\mathbb{P}}^*
\bigl(q_S^*T^*(X/S)\otimes\mathcal O(1)\bigr)$.
Together with \eqref{eq:auto-048}, we have
\begin{equation}\label{eq:auto-053}
\overline{d_{/S'}f'}^{\,!}
\bigl((\pi_{X'/S'})_*CC(i_X^*\mathcal F)\bigr)_{\mathcal O(1)}
 =
(-1)^{\dim X'-\dim X}
i_{\mathbb{P}}^!
\left[
\overline{d_{/S}f}^{\,!}
\bigl((\pi_{X/S})_*CC(\mathcal F)\bigr)_{\mathcal O(1)}
\right].
\end{equation}
Applying \eqref{eq:projective-relative-expression} and $i_X^!(q_{S})_*=(q_{S'})_*i_{\mathbb{P}}^!$
we have
\begin{equation}\label{eq:auto-054}
cc^{Z'}_{X'/Y'/S'}(i_X^*\mathcal F)
=
(-1)^{s'+\dim X'-\dim X-s}
\,i_X^!cc^Z_{X/Y/S}(\mathcal F).
\end{equation}
Since
$
\dim X'-\dim X=s'-s,
$
we have \eqref{eq:pbGNA}.
\end{proof}

 In particular, if $i_X$ is flat of pure relative dimension
$s'-s$, then
\begin{equation}\label{eq:auto-055}
{\rm pr}^{-1}SS(\mathcal F)\cap\ker(di_X)\subset0(X'),
\end{equation}
so $i_X$ is $SS(\mathcal F)$-transversal; moreover, for every
irreducible component $C\subset SS(\mathcal F)$,
\begin{equation}\label{eq:auto-056}
\dim(C\times_XX')
=
\dim C+s'-s
=
\dim X+s'-s
=
\dim X',
\end{equation}
hence $i_X$ is properly $SS(\mathcal F)$-transversal and the pullback formula applies.

\begin{proposition}\label{prop:pfGNA}
Let $p:X\to X'$ be a proper morphism in ${\rm Sm}_k$, with $X'$ of
pure dimension, and let
$f':X'\to Y$ satisfy $f=f'\circ p$. Let $Z'\subseteq X'$ be closed,
with $Z\subseteq p^{-1}(Z')$. Assume that
$ \dim p_\circ SS(\mathcal F)\leq\dim X'$,
then $g\circ f'$ is $SS(Rp_*\mathcal F)$-transversal, and $f'$ is
$SS(Rp_*\mathcal F)$-transversal outside $Z'$. Moreover,
\begin{equation}\label{eq:pfgNA}
 p_*cc^Z_{X/Y/S}(\mathcal F)
 =cc^{Z'}_{X'/Y/S}(Rp_*\mathcal F)
 \quad\text{in }{\rm CH}_s(Z').
\end{equation}
\end{proposition}
\begin{proof}
Proper pushforward of singular support gives
$SS(Rp_*\mathcal F)\subseteq p_\circ SS(\mathcal F)$
\cite[Lemma~2.2(ii)]{Bei16}. A covector in $p_\circ SS(\mathcal F)$
lifts to a covector whose pullback by $dp$ belongs to $SS(\mathcal F)$.
The hypotheses on $h$ and $f$ therefore give
\begin{equation}\label{eq:auto-057}
\begin{split}
 (df')^{-1}p_\circ SS(\mathcal F)
 \subseteq 0(X')\cup(T^*Y\times_YZ'),\\
 (df')^{-1}p_\circ SS(\mathcal F)
 \cap(T^*S\times_SX')\subseteq 0(X').
\end{split}
\end{equation}
Here $0(X')$ is the zero section of $T^*Y\times_YX'$.
This proves both transversality assertions, and also gives a common
projective support on which the quotient bundle is defined.

Abe's proper-pushforward theorem
\cite[Introduction, Theorem; \S6.6]{Abe21} gives
\begin{equation}\label{eq:pfGNAc1}
 CC(Rp_*\mathcal F)
 =p_!CC(\mathcal F)
 ={\rm pr}_{2*}(dp)^!CC(\mathcal F)\qquad
 \in\operatorname{CH}_{\dim X'}(p_{\circ}SS(\mathcal F))\left[\frac{1}{\operatorname{char}k}\right]
\end{equation}
where $\operatorname{char}k>0$. 
Let $q'$ be the projective projection over $X'$, and let
\begin{equation}\label{eq:auto-058}
p_{\mathbb{P}}:\mathbb{P}(T^*Y\times_YX)
 \longrightarrow\mathbb{P}(T^*Y\times_YX')
\end{equation}
be the base change of $p$. We have natural maps
\begin{equation}\label{eq:auto-059}
\begin{split}
 dp\otimes1:
 q^*(T^*X'\times_{X'}X)\otimes\mathcal O(1)
 &\longrightarrow q^*T^*X\otimes\mathcal O(1),\\
 \rho:
 q^*(T^*X'\times_{X'}X)\otimes\mathcal O(1)
 &\longrightarrow q'^*T^*X'\otimes\mathcal O(1).
\end{split}
\end{equation}
The second map is proper. 
Let
\begin{equation}\label{eq:auto-060}
D\subset q'^*T^*X'\otimes\mathcal O(1)
\end{equation}
be the $\mathcal O(1)$-twist of the pullback of $p_\circ SS(\mathcal F)$, then
$
\dim D\leq \dim X'+\dim Y-1
$.
On an open subset trivializing $\mathcal{O}(-1)$ we have
\begin{equation}\label{eq:pbcal}
  \rho_*(dp\otimes1)^!CC(\mathcal F)_{\mathcal O(1)}
=
CC(Rp_*\mathcal F)_{\mathcal O(1)}
\qquad\text{in }
\operatorname{CH}_{\dim X'+\dim Y-1}(D)\left[\frac{1}{\operatorname{char}k}\right].  
\end{equation}
Since
$
\operatorname{CH}_{\dim X'+\dim Y-1}(D)
=
Z_{\dim X'+\dim Y-1}(D)
$
and inverting $\operatorname{char}k$ is injective on the free abelian group 
$Z_{\dim X'+\dim Y-1}(D)$,
the local equalities \eqref{eq:pbcal} are equalities of cycles and hence glue globally
into an equality in $Z_{\dim X'+\dim Y-1}(D)$.
Let $p_{\mathbb{P}}$ be the base change of $p$. The pullback
of $\overline{df'}$ is a section
\begin{equation}\label{eq:auto-061}
p_{\mathbb{P}}^*\overline{df}:
\mathbb{P}(T^*Y\times_YX)
\longrightarrow
q^*(T^*X'\times_{X'}X)\otimes\mathcal O(1),
\end{equation}
and
\begin{equation}\label{eq:auto-062}
(dp\otimes1)\circ p_{\mathbb{P}}^*\overline{df'}=\overline{df},
\qquad
\rho\circ p_{\mathbb{P}}^*\overline{df'}=\overline{df'}\circ p_{\mathbb{P}}.
\end{equation}
Therefore, by compatibility of refined Gysin pullback with proper pushforward \cite[Theorem~6.2]{Ful98},
\begin{equation}\label{eq:auto-063}
\begin{aligned}
(p_{\mathbb{P}})_*\overline{df}^!CC(\mathcal F)_{\mathcal O(1)}
&=
(p_{\mathbb{P}})_*
(p_{\mathbb{P}}^*\overline{df'})^!
(dp\otimes1)^!CC(\mathcal F)_{\mathcal O(1)}
\\
&=
\overline{df'}^!\,
\rho_*(dp\otimes1)^!CC(\mathcal F)_{\mathcal O(1)}
\\
&=
\overline{df'}^!CC(Rp_*\mathcal F)_{\mathcal O(1)}.
\end{aligned}
\end{equation}
Combining with
\begin{equation}\label{eq:auto-064}
p_{\mathbb{P}}^*
\left(
\frac{q'^*f'^*T^*(Y/S)}{\mathcal O(-1)}
\right)
=
\frac{q^*f^*T^*(Y/S)}{\mathcal O(-1)},
\end{equation}
we have
\begin{equation}\label{eq:auto-065}
\begin{aligned}
p_*cc^Z_{X/Y/S}(\mathcal F)
&=
(-1)^s q'_*
\left[
c_{r-1}\left(
\frac{q'^*f'^*T^*(Y/S)}{\mathcal O(-1)}
\right)
\cap
\overline{df'}^!CC(Rp_*\mathcal F)_{\mathcal O(1)}
\right]
\\
&=
cc^{Z'}_{X'/Y/S}(Rp_*\mathcal F),
\end{aligned}
\end{equation}
which proves the assertion.
\end{proof}

\begin{corollary}[Conductor formula]\label{cor:projective-conductor}
Let $S={\rm Spec}(k)$, let $Y$ be a smooth connected curve, and let
$f:X\to Y$ be proper. Let $T\subset Y$ be a finite set of closed points,
and put $Z=f^{-1}(T)$. Assume that $f$ is $SS(\mathcal F)$-transversal
outside $Z$. Then
\begin{equation}\label{eq:SCF}
 f_*cc^Z_{X/Y/k}(\mathcal F)
 =-\sum_{y\in T}a_y(Rf_*\mathcal F)[y]
 \quad\text{in }{\rm CH}_0(T),
\end{equation}
where $a_y$ is the Artin conductor, computed after derived reduction of
the coefficient ring to its residue field.
\end{corollary}
\begin{proof}
Combining Theorem~\ref{thm:SMF} and Proposition~\ref{prop:pfGNA} we get the desired result.
\end{proof}


\section{Non-isolated Milnor formula in quasi-affine case}
\subsection{}\label{qafmil}
Let $k$ be an algebraically closed field of characteristic $p > 0$, $\Lambda$ be a finite local ring such that the characteristic of the residue field is invertible in $k$. Let $X$ be a smooth separated scheme of finite type over $k$ purely of dimension $d$, let $f : X \rightarrow \A^1_k$ be a morphism and let $\mathcal F\in D_{\rm ctf}(X,\Lambda)$.
Let $t$ be the standard coordinate on $\A^1_k$. The form $dt$ defines 
\begin{equation}\label{eq:display-1036}
df: X\longrightarrow X\times_{\A^1_k}T^*\A^1_k\longrightarrow T^{\ast}X
\end{equation}
which gives a section of $T^{\ast}X$. Put $Z=Z_{X,f}\coloneq (df)^{-1}SS(\mathcal{F})$, \cite[Lemma~4.2(5)]{Sai17a} implies $f$ is universally locally acyclic relatively to $\mathcal{F}$ in $X\setminus Z $.

\begin{theorem}\label{thmqafmil}
    Consider the previous assumptions. Assume that $X$ is quasi-affine. Then we have
    \begin{equation}\label{eq:qafmil}
\widetilde{C}^Z_{X/\A^1_k/k}(\mathcal{F})=\widetilde{{\rm cl}}\left( {df}^! CC(\mathcal F)\right).
\end{equation}
\end{theorem}
By Proposition \ref{prop:geoNA-affine-line}, $cc^Z_{X/\A^1_k/k}(\mathcal F)=df^!CC(\mathcal F)$ in $\mathrm{CH}_0(Z)$. Hence Theorem \ref{thmqafmil} verifies Conjecture \ref{conj:MFforNI} in the quasi-affine case, with $Y=\A^1_k$, $S=\Spec k$.

We devote the rest of this section to the proof of the theorem.

\subsection{}
\label{spgysinpull}
We first recall Fulton's construction of specialization over a discrete valuation ring \cite[Chapter~20]{Ful98}. 
Let $R$ be a discrete valuation ring, let $S=\Spec R$ with $j:\eta\hookrightarrow S$, $i:s\hookrightarrow S$ be its generic point and closed point, respectively. Let $u$ be a uniformizer of $R$.

Define the dimension function by setting $\delta_S(s)=0$, $\delta_S(\eta)=1$ and
\begin{equation}\label{eq:display-1058}
 \delta_T(x)=\delta_S(p(x))
        +\operatorname{trdeg}_{\kappa(p(x))}\kappa(x)
\end{equation}
for a morphism $p:T\to S$  locally of finite type with $x\in T$.
For an integral closed subscheme $V\subset T$, its $S$-dimension is defined to be $\delta_T(\xi_V)$ where $\xi_V$ is the generic point of $V$.  
Let $Z_m(T/S)$ be the free abelian group on integral closed
subschemes of $S$-dimension $m$, we may define rational equivalence in the usual way:
it is generated by $\operatorname{div}_V(h)$ for rational functions $h$ on
integral closed $V$ of $S$-dimension $m+1$.  Put
${\rm CH}_m(T/S)=Z_m(T/S)/\!\sim_{\mathrm{rat}}$.  
Let $j_T: T_\eta \rightarrow T$, $i_T: T_s \rightarrow T$ be the base changes of $j$ and $i$, respectively. Then, with a slight abuse of notation, we have
\begin{equation}\label{eq:relative-degrees}
 j_T^*: {\rm CH}_{m+1}(T/S)\longrightarrow{\rm CH}_m(T_\eta),
 \qquad
 i_T^!: {\rm CH}_{m+1}(T/S)\longrightarrow{\rm CH}_m(T_s).
\end{equation}
The second map is the refined Gysin map obtained by the regular immersion $i$. 
Moreover, we have the localization exact sequence
\begin{equation}\label{eq:rellocseq}
 {\rm CH}_{r+1}(T_s)\xrightarrow{{i_T}_*}{\rm CH}_{r+1}(T/S)
 \xrightarrow{j_T^*}{\rm CH}_r(T_\eta)\longrightarrow0.
\end{equation}
For any $\alpha\in{\rm CH}_r(T_\eta)$, choose a lift $\widetilde \alpha \in {\rm CH}_{r+1}(T/S)$ by
\eqref{eq:rellocseq} and define
\begin{equation}\label{eq:spch-lift}
 \spcy(\alpha)=i_T^!(\widetilde \alpha)\in{\rm CH}_r(T_s).
\end{equation}
One may check $\spcy(\alpha)$ is independent of the lift.  Indeed, two lifts differ by some
${i_T}_*\xi$, and the refined self-intersection formula gives
\begin{equation}\label{eq:display-1091}
 i_T^!{i_T}_*\xi
 =c_1\bigl(p_s^*N_{s/S}\bigr)\cap\xi=0,
\end{equation}
because $N_{s/S}\simeq\mathcal O_s$.  In particular,
\begin{equation}\label{eq:spch-relative-lift}
 \spcy(j_T^*\gamma)=i_T^!\gamma
\end{equation}
for $\gamma\in{\rm CH}_{r+1}(T/S)$.
For integral subscheme $V\subset T$ such that $V_{\eta}\neq \varnothing$ (in this case we call $V$ \emph{horizontal}), $\spcy([V_{\eta}])$ is the divisor $\operatorname{div}_V(u)$, with the usual length multiplicities along the special fiber. 

\subsection{}
\label{secgyssp}
Let $B$ be a smooth scheme separated of finite type over $S$.  Let $E\to B$ be a vector
bundle of rank $r$, let $\sigma:B\to E$ be a section, let $C$ be a
closed subscheme of $E$, and put $Z=B\times_{E}C$.
Since $\sigma$ is a regular immersion of codimension $r$,
there is a supported refined Gysin map
\begin{equation}\label{eq:display-1108}
 \sigma^!: {\rm CH}_{m+1}(C/S)\longrightarrow{\rm CH}_{m-r+1}(Z/S).
\end{equation}

\begin{lemma}
\label{gyssplem}
Consider the assumptions in \S\ref{secgyssp}.
Let $\gamma\in{\rm CH}_{m+1}(C/S)$.  Define
\begin{equation}\label{eq:display-1115}
 \gamma_\eta=j_C^*\gamma\in{\rm CH}_m(C_\eta),
 \qquad
 \gamma_s=i_C^!\gamma\in{\rm CH}_m(C_s).
\end{equation}
Then we have
\begin{equation}\label{eq:gysspge}
\begin{aligned}
j_Z^*(\sigma^!\gamma)&=\sigma_\eta^!\gamma_\eta
       &&\text{in }{\rm CH}_{m-r}(Z_\eta),
\end{aligned}
\end{equation}
\begin{equation}\label{eq:gysspsp}
\begin{aligned}
\spcy\bigl(\sigma_\eta^!\gamma_\eta\bigr)
   &=\sigma_s^!\gamma_s
       &&\text{in }{\rm CH}_{m-r}(Z_s).
\end{aligned}
\end{equation}
\end{lemma}
\begin{proof}
	   Refined Gysin pullback exchanges with flat base change
$\eta\rightarrow S$, hence we have \begin{equation}\label{eq:auto-066}
j^*_Z\circ \sigma^!=\sigma^!_{\eta}\circ j^*_C,
\end{equation}
which proves \eqref{eq:gysspge}. 

For the special fiber, we have
\begin{equation}\label{eq:display-1134}
\spcy(\sigma_\eta^!\gamma_\eta)
        \overset{\eqref{eq:gysspge}}{=}\spcy(j_Z^*\sigma^!\gamma)
        \overset{\eqref{eq:spch-relative-lift}}{=}i_Z^!\sigma^!\gamma
        \overset{\text{\cite[\S6.4]{Ful98}}}{=}\sigma_s^!i_C^!\gamma.
\end{equation}

\end{proof}

\begin{corollary}
\label{flatgyssp}
Consider the assumptions in Lemma~\ref{gyssplem}.
Suppose
\begin{equation}\label{eq:display-1148}
 \alpha=\sum_a n_a[C_a]
\end{equation}
is a cycle supported on $C$, with components $C_a\subset E$ flat over $S$ of pure relative dimension $m$,
and put $q=m-r\geq0$.  For $t\in\{\eta,s\}$ define
\begin{equation}\label{eq:display-1153}
 \alpha_t:=\sum_{a:(C_a)_t\ne\varnothing}n_a[(C_a)_t],
\end{equation}
with the usual scheme-theoretic multiplicities at its generic points.
Then
\begin{equation}\label{eq:flat-gysin-sp}
 \spcy\bigl(\sigma_\eta^!\alpha_\eta\bigr)
 =\sigma_s^!\alpha_s\qquad\text{in }{\rm CH}_q(Z_s).
\end{equation}

\end{corollary}

\begin{proof}
Since each $C_a$ is flat over $S$ with pure $m$-dimension fibers,
we have a natural lift 
$\widetilde\alpha=\sum_a n_a[C_a]$ in ${\rm CH}_{m+1}(C/S)$.  Flat pullback gives
$j_C^*\widetilde\alpha=\alpha_\eta$, while Gysin pullback along $s\rightarrow S$
gives $i_C^!\widetilde\alpha=\alpha_s$.  The consequence follows from Lemma~\ref{gyssplem}.
\end{proof}

\subsection{}\label{spclcomm}
In sections \ref{spclcomm}--\ref{clspcomm}, let $S=\Spec R$ be a strictly henselian trait with generic point $\eta$ and closed point $s$, and let $u\in R$ be a uniformizer. Let $f:X\rightarrow S$ be a smooth separated morphism of finite type and pure relative dimension $d$. Let $k: Z\rightarrow X$ be a closed subscheme. Let $j_X: X_\eta \rightarrow X$, $i_X: X_s \rightarrow X$ be the generic and special fiber, respectively, and $b_X:X_{\overline \eta} \rightarrow X_\eta$ be the base change to the algebraic closure $\overline \eta$ . We use the convention of nearby cycle functor in \cite[\S3.4]{YZ22}:
\begin{equation}\label{eq:display-1174}
R\Psi_f=i^*_X {Rj_X}_* {b_X}_* {b_X}^* :D^+(X_{\eta},\Lambda)\rightarrow D^+(X_{s},\Lambda).
\end{equation}

For $0\leq r\leq d$, $c=d-r$ and $t\in\{\eta,s\}$,  we have the refined cycle class map
\begin{equation}\label{eq:refcyc}
 {\rm cl}_t:{\rm CH}_r(Z_t)
 \longrightarrow H^{2c}_{Z_t}(X_t,\Lambda(c))
 \simeq H^{-2r}_{Z_t}(X_t,\mathcal K_{X_t/t}(-r))
\end{equation}
by Riou \cite[Exposé XVI, Definition~2.3.1, Proposition~2.3.7,  Definition~2.4.1]{Rio14}.

Using the method of \cite[\S3.5]{YZ22}, we have a natural specialization map on \'etale cohomology groups
\begin{equation}\label{eq:etsp}
 \spet:
 H^{-2r}_{Z_\eta}(X_\eta,\mathcal K_{X_\eta/\eta}(-r))
 \longrightarrow
 H^{-2r}_{Z_s}(X_s,\mathcal K_{X_s/s}(-r)).
\end{equation}

\begin{theorem}
\label{spcycommthm}
For every $0\leq r\leq d$ and $c=d-r$, the square
\begin{equation}\label{eq:all-degree-square}
\begin{CD}
 {\rm CH}_r(Z_\eta)
 @>{{\rm cl}_\eta}>> H^{2c}_{Z_\eta}(X_\eta,\Lambda(c))\\
 @V{\spcy}VV @VV{\spet}V\\
 {\rm CH}_r(Z_s)
 @>{{\rm cl}_s}>> H^{2c}_{Z_s}(X_s,\Lambda(c))
\end{CD}
\end{equation}
commutes.
\end{theorem}
We postpone the proof of Theorem~\ref{spcycommthm} to \S\ref{pf48}.
\begin{lemma}
\label{etextsp}
Let $W\subset Z$ be closed and
$ \gamma\in H^{2a}_W(X,\mathcal K_{X/S}(a))$,
we have
\begin{equation}\label{eq:extendable-class}
 \spet(j_X^*\gamma)=i_X^*\gamma
 \quad\mathrm{in}\quad H^{2a}_{W_{s}}(X_{s},\mathcal K_{X_{s}/{s}}(a)).
\end{equation}
\end{lemma}

\begin{proof}
Let $\ell:W\rightarrow X$ be the closed immersion.  By definition, $\gamma$ is equivalent to a morphism
$\gamma:\ell_*\Lambda_W\to \mathcal K_{X/S}(a)[2a]$ and we have
\begin{equation}\label{eq:canw}
\ell_{s*}\Lambda_{W_s}\simeq i^*_X \ell_* \Lambda_W
 \longrightarrow
 R\Psi_f j^*_X \ell_* \Lambda_W \simeq 
 R\Psi_f(\ell_{\eta*}\Lambda_{W_\eta}),
\end{equation}
\begin{equation}\label{eq:cank}
i_X^*\mathcal K_{X/S}(a)[2a]
 \longrightarrow
 R\Psi_f(j_X^*\mathcal K_{X/S})(a)[2a]
 \simeq
 R\Psi_f(\mathcal K_{X_\eta/\eta})(a)[2a].
\end{equation}
by proper base change.
Moreover, we have the following commutative square
\begin{equation}\label{eq:nearby-naturality-square}
\begin{CD}
 \ell_{s*}\Lambda_{W_s}
 @>{\eqref{eq:canw}}>>
 R\Psi_f(\ell_{\eta*}\Lambda_{W_\eta})\\
 @V{i_X^*\gamma}VV
 @VV{R\Psi_f(j^*_X\gamma)}V\\
 i_X^*\mathcal K_{X/S}(a)[2a]
 @>{\eqref{eq:cank}}>>
 R\Psi_f(\mathcal K_{X_\eta/\eta})(a)[2a]
\end{CD}
\end{equation}
by the naturality of the nearby cycle functor.
By \cite[(3.67)]{YZ22}, we have a natural map
\begin{equation}\label{eq:display-1255}
\theta_f:
 R\Psi_f(\mathcal K_{X_\eta/\eta})\longrightarrow \mathcal K_{X_s/s}.
\end{equation}
Combining the definition of specialization in
\cite[(3.68), (3.69)]{YZ22} with \eqref{eq:nearby-naturality-square}, we have
\begin{equation}\label{eq:etspcal}
\spet(j_X^*\gamma)
 =\theta_f(a)[2a]\circ R\Psi_f(j^*_X\gamma)\circ\eqref{eq:canw}
 =\theta_f(a)[2a]\circ \eqref{eq:cank}\circ i^*_X\gamma.
\end{equation}
Since $f$ is smooth, local acyclicity and purity imply 
$\theta_f(a)[2a]\circ\eqref{eq:cank}$ is precisely the isomorphism
\begin{equation}\label{eq:display-1268}
i^*_X \mathcal K_{X/S}(a)[2a]\xrightarrow{\sim}\mathcal K_{X_s/s}(a)[2a],
\end{equation}
hence the right hand side of \eqref{eq:etspcal} is exactly $i^*_X \gamma$.

\end{proof}

\subsection{}\label{proofspcyc}
We use the following result of SGA~$4\frac12$ \cite{SGA4h}.  Assume $X'$ is smooth over $S$ of pure relative dimension $N$ ,
a closed subscheme $Z'\subset X'$ has fiber codimension at least $c$ and $\mathcal K$ is a
bounded coherent perfect $\mathcal O_{X'}$-complex supported on $Z'$ and of finite
Tor-dimension over $S$. Then \cite[Cycle, 2.3.3]{SGA4h} constructs a general cycle class
\begin{equation}\label{eq:display-1282}
 {\rm cl}_S(Z',\mathcal K)\in H_{Z'}^{2c}(X',\Lambda(c)).
\end{equation}
Its formation commutes with arbitrary base change, and when $S$ is a field, it agrees with the usual supported cycle class and factors through the Chow group. See \cite[Cycle, Lemme~2.3.7, 2.2.10]{SGA4h}.
\subsection{}\label{pf48}
We now prove Theorem~\ref{spcycommthm}.
\begin{proof}[Proof of Theorem~\ref{spcycommthm}]
Both paths in \eqref{eq:all-degree-square} are homomorphisms on Chow groups, so
it is enough to treat a prime cycle $[W_\eta]$ supported on $Z$. 

Let $W\subset Z$ be the scheme-theoretic closure of $W_\eta$ which is
integral.  Multiplication by $u$ is injective in every affine coordinate ring
of $W$, so $W$ is torsion-free over 
$\Spec R$, hence flat over $S$ by \cite[Tag~0539]{Stacks}.  If the special fiber $W_s$ is
nonempty, flatness implies $\dim W_s=r$. Set $c=d-r$, the codimension. We have
\begin{equation}\label{eq:prime-sp-cycle}
 \spcy[W_\eta]
 =\sum_{\xi\in(W_s)^{(0)}}
  m_\xi[\overline{\{\xi\}}],
 \qquad
 m_\xi=\operatorname{length}_{\mathcal O_{ W,\xi}}
       (\mathcal O_{ W,\xi}/u)
\end{equation}
by the previous description of $\spcy([V])$ of a horizontal integral subscheme $V$. If $W_s$ is empty, the right side of \eqref{eq:prime-sp-cycle} is zero.  
In the following we assume $W_s\neq \varnothing$.

The scheme $X$ is regular and Noetherian because $S$ is a
regular trait and $f$ is smooth of finite type; see \cite[Tags~0H7S and 0FDC]{Stacks}.  
Hence the coherent module $\mathcal O_{W}$ is perfect as an $\mathcal O_X$-complex by \cite[Tag~0FDC]{Stacks}. 
Moreover, flatness of $W$ over $S$ implies $\mathcal O_{W}$ is finite Tor-dimension over $S$.  Thus by \cite[Cycle, 2.3.3]{SGA4h} we may form
\begin{equation}\label{eq:relcyc}
{\rm cl}_S(W,\mathcal O_{W})
 \in H_{W}^{2c}(X,\Lambda(c))
 \simeq H^{-2r}_{W}(X,\mathcal K_{X/S}(-r)).
\end{equation}
On the generic fiber, the length of
$\mathcal O_{W_\eta}$ is one.  On the special fiber, flatness gives
$Li_X^*\mathcal O_{W}=\mathcal O_{W_s}$, whose generic
length at $\xi$ is the integer $m_\xi$ in
\eqref{eq:prime-sp-cycle}.  Therefore by \eqref{eq:prime-sp-cycle} and \cite[Cycle, Lemme~2.3.7, Th\'eor\`eme~2.3.8(i)] {SGA4h}, we have
\begin{equation}\label{eq:relative-class-generic}
\begin{aligned}
j_X^*{\rm cl}_S(W,\mathcal O_{W})&={\rm cl}_\eta[W_\eta],
\end{aligned}
\end{equation}
\begin{equation}\label{eq:relative-class-special}
\begin{aligned}
i_X^*{\rm cl}_S(W,\mathcal O_{W})
 &=\sum_{\xi}m_\xi{\rm cl}_s[\overline{\{\xi\}}]
 ={\rm cl}_s\bigl(\spcy[W_\eta]\bigr).
\end{aligned}
\end{equation}
Applying Lemma~\ref{etextsp}, we have
\begin{equation}\label{eq:display-1329}
 \spet{\rm cl}_\eta[W_\eta]
 =\spet(j_X^*{\rm cl}_S(W,\mathcal O_{W}))
 =i_X^*{\rm cl}_S(W,\mathcal O_{W})
 ={\rm cl}_s\bigl(\spcy[W_\eta]\bigr).
\end{equation}
This proves the commutativity of \eqref{eq:all-degree-square} on prime cycles.  Fulton
\cite[\S20.3]{Ful98} proves that $\spcy$ respects rational equivalence, and
the supported cycle class map factors through Chow groups by \cite[Cycle, \S2.2]{SGA4h}.
The equality therefore descends to ${\rm CH}_r(Z_\eta)$.
\end{proof}

\begin{corollary}
\label{clspcomm}
In the notation of Lemma~\ref{gyssplem}, suppose in addition that $S$
is a strictly henselian trait  and $B\rightarrow S$ is smooth of pure relative dimension $d$.  Put $0\leq q=m-r\leq d$, then in $H^{-2q}_{Z_s}(B_s,\mathcal K_{B_s/s}(-q))
 \simeq H^{2(d-q)}_{Z_s}(B_s,\Lambda(d-q))$
 we have
\begin{equation}\label{eq:clspcomm}
 \spet{\rm cl}_\eta(\sigma_\eta^!\gamma_\eta)
 ={\rm cl}_s(\sigma_s^!\gamma_s).
\end{equation}
More generally, let
\begin{equation}\label{eq:display-1357}
 \widetilde{\alpha}\in Z_{m+1}(C/S)
\end{equation}
be any cycle, denote its class in ${\rm CH}_{m+1}(C/S)$ by $\gamma$, and set
\begin{equation}\label{eq:display-1361}
 \alpha_\eta:=j_C^*\gamma\in{\rm CH}_m(C_\eta),
 \qquad
 \alpha_s:=i_C^!\gamma\in{\rm CH}_m(C_s).
\end{equation}
Then \eqref{eq:clspcomm} has the form
\begin{equation}\label{eq:display-1367}
 \spet{\rm cl}_\eta\!\left(\sigma_\eta^!\alpha_\eta\right)
 =
 {\rm cl}_s\!\left(\sigma_s^!\alpha_s\right).
\end{equation}

\end{corollary}

\begin{proof}
Apply Theorem~\ref{spcycommthm} to
$\sigma_\eta^!\gamma_\eta$ we have
\begin{equation}\label{eq:display-1380}
 \spet{\rm cl}_\eta(\sigma_\eta^!\gamma_\eta)
 ={\rm cl}_s\!\left(\spcy(\sigma_\eta^!\gamma_\eta)\right).
\end{equation}
Now use \eqref{eq:gysspsp} we get the desired conclusion.
\end{proof}

   \begin{proposition}\label{qafeq}
      For a separated scheme $X$ of finite type over $k$, the following conditions are equivalent: 
      \begin{enumerate}
          \item there is a finite-dimensional subspace $V\subset\Gamma(X,\mathcal O_X)$ such that
             \begin{equation}\label{eq:exact-generation}
             V\otimes_k\mathcal O_X\longrightarrow\Omega^1_{X/k},
              \qquad v\otimes g\longmapsto g\cdot dv,
              \end{equation}
        is surjective.
          \item  $X$ is quasi-affine.
      \end{enumerate}
   \end{proposition}

\begin{proof}
First assume condition (2). By \cite[Tag~04II]{Stacks}, there is a closed immersion $i:X\rightarrow U$, where $U$ is an open subscheme of $\A^n_k$. Therefore, the natural homomorphism 
$i^* \Omega^1_{U/k}\rightarrow \Omega^1_{X/k}$
is a surjection. The consequence follows from 
\begin{equation}\label{eq:display-1408}
i^* \Omega^1_{U/k} \simeq \bigoplus^n_{j=1}\mathcal{O}_X dx_j.
\end{equation}

Now we show that (1) implies (2). Let $V$ be the subspace of $\Gamma(X,\mathcal{O}_X)$ satisfying the condition, and fix a basis $g_1,\ldots,g_N$. Put
\begin{equation}\label{eq:display-1413}
g=(g_1,\ldots,g_N):X\to\A^N_k.
\end{equation}
In the cotangent sequence 
\begin{equation}\label{eq:ctgseq}
 g^*\Omega^1_{\A^N_k/k}\longrightarrow\Omega^1_{X/k}
 \longrightarrow\Omega^1_{X/\A^N_k}\longrightarrow0,
\end{equation}
the first arrow is $dx_i\mapsto dg_i$, where $x_1,\cdots,x_N$ are the coordinates of $\A^N_k$. The surjectivity condition hence implies $\Omega^1_{X/\A^N_k}=0$ and $g$ is unramified. An unramified morphism is locally quasi-finite; because $g$ is also quasi-compact, it is quasi-finite. See \cite[Tags~02V5 and 01TJ]{Stacks}.

The diagonal map $X\rightarrow X\times_{\A^N_k}X$ is a base change of the diagonal $X\rightarrow X\times_k X$; hence $g$ is separated.  We may apply Zariski's main theorem to get the factorization
\begin{equation}\label{eq:display-1424}
 X\overset{j}\rightarrow\overline X
 \overset{\overline g}{\rightarrow}\A^N_k,
\end{equation}
where $j$ is an open immersion and $\overline g$ is finite.
Hence $\overline X$ is affine, and its open subscheme
$X$ is quasi-affine.

\end{proof}
\subsection{}
In the following we use the notations of Theorem~\ref{thmqafmil}. Let $V$ be a subspace of $\Gamma(X,\mathcal O_X)$ obtained by Proposition~\ref{qafeq}, with $N=\dim V$, $\mathbb{V}:= \Spec(\operatorname{Sym}_k V^{\vee})\simeq \A^N_k$. Define
\begin{equation}\label{eq:amap}
 a:X\times_k \mathbb{V}\longrightarrow T^{\ast}X,
 \qquad (x,g)\longmapsto df_x+dg_x.
\end{equation}
\begin{lemma}
    \label{genpert}
There is a dense open subset $\VV^\circ\subset \VV$ such that
$(df+dg)^{-1}(SS(\mathcal{F}))$ is finite for every $g\in \VV^\circ(k)$.  One may choose
$g\in \VV(k)$ such that the generic point of the line $k\cdot g$ belongs to
$\VV^\circ$.
\end{lemma}

\begin{proof}
The map $a$ lies over the identity of
$X$, and its vertical differential is calculated by the surjective evaluation
$V\to T_x^*X$, $g\mapsto dg_x$. Hence $a$ is smooth of relative dimension $N-d$.  Since each component of $SS(\mathcal{F})$ has
dimension $d$,
every component of
\begin{equation}\label{eq:display-1453}
 Q=a^{-1}(SS(\mathcal{F}))\subset X\times \VV
\end{equation}
has dimension $N$.
Let $q:Q\to \VV$ be the projection, let $Q_1,\cdots,Q_m$ be the irreducible components of $Q$.  Put 
	\begin{equation}\label{eq:display-1459}
D:=\bigcup_{\overline{q(Q_i)}\subsetneq \VV} \overline{q(Q_i)}\subsetneq \VV,
\end{equation}
$U_D:=\VV\setminus D$ is the complement. Then for any $v\in U_D$, and $Q_i$ with $\overline{q(Q_i)}\subsetneq \VV$, we have $(Q_i)_v=\varnothing$.

For $Q_j$ such that $Q_j$ dominates $\VV$, i.e. $\overline{q(Q_j)}= \VV$, we have 
\begin{equation}\label{eq:display-1465}
\dim (Q_j)_{\eta}=\dim Q_j - \dim \VV =0
\end{equation}
where $\eta$ is the generic point of $\VV$.
By a generic fiber dimension result \cite[Tag~05F7]{Stacks}, $\dim (Q_j)_u = 0$ for all $u$ in an open subscheme $U_j$ of $\VV$. Therefore,
\begin{equation}\label{eq:display-1470}
\VV^\circ:= (\bigcap_{Q_j\mathrm{\ dominates\ }\VV}U_j)\cap U_D
\end{equation}
has the required property.

Choose $v\in \VV^\circ(k)$.  The line $k\cdot v$ is not contained in
$\VV\setminus \VV^\circ$, because it contains $v$; hence its generic point belongs to
$\VV^\circ$.  Now $g=v$ verifies the demand of the lemma.
\end{proof}

\subsection{}
\label{qafpfdef}

Choose $g$ as in Lemma~\ref{genpert}.  Let $R$ be the strict henselization of
$k[u]_{(u)}$, set $S=\Spec R$, write $K=\operatorname{Frac}(R)$, and put
$\eta^{\operatorname{perf}}=\Spec K^{\operatorname{perf}}$. The natural inclusion $k\rightarrow R$ gives $S\rightarrow \Spec k$. Let $\eta=\Spec K$, $s=\Spec k$ be the generic and special point of $S$, respectively. 
Define
\begin{equation}\label{eq:display-1488}
 \mathcal X=X\times_kS,\qquad
 \mathcal Y=\A^1_S,\qquad
 \widetilde{\mathcal F}=\operatorname{pr}_X^*\mathcal F,
\end{equation}
\begin{equation}\label{eq:display-1493}
 \phi:\mathcal X\rightarrow \mathcal Y,\qquad\phi(x,u)=f(x)+u g(x),
\end{equation}
where $\operatorname{pr}_X:\mathcal X\to X$ is the projection.  The differential $d\phi$ gives a section
\begin{equation}\label{eq:sigma}
 \sigma=d\phi=df+u\,dg:
 \mathcal X\longrightarrow T^{\ast}(\mathcal X/S)=T^{\ast}X\times_kS.
\end{equation}
Put $\mathcal W=\sigma^{-1}(SS(\mathcal F)\times_kS)$,
then by Lemma~\ref{genpert},
\begin{equation}\label{eq:display-1503}
 \mathcal W_s=Z=Z_{X,f},\qquad
 \mathcal W_\eta\text{ is finite over }K,
\end{equation}
and hence  $\mathcal W_{\eta^{\operatorname{perf}}}$ is finite over $K^{\operatorname{perf}}$.

For $t\in\{\eta,s\}$, 
$\phi_t$ is $(SS(\mathcal F)\times_k S)_t$-transversal outside $\mathcal W_t$.
For $t=s$, \cite[Lemma~4.2(5)]{Sai17a} implies that
$\phi_s$ is universally locally acyclic relatively to
$\widetilde{\mathcal F}_s$ outside $\mathcal W_s$.
The projection
\begin{equation}\label{eq:display-1517}
 \operatorname{pr}_S:\mathcal X=X\times_k S\rightarrow S
\end{equation}
is universally locally acyclic relatively to
$\widetilde{\mathcal F}$, and
\begin{equation}\label{eq:display-1522}
 R\Psi_{\operatorname{pr}_S}(\widetilde{\mathcal F}_\eta)
 \simeq \widetilde{\mathcal F}_s
 \simeq \mathcal F.
\end{equation}

\subsection{}
Retain the notation of \S\ref{qafpfdef}. Initially, the generic fiber is over $K$ which is not perfect in general; however, the construction of characteristic cycle in \cite{Sai17a} needs the base field to be perfect. In the following, we build a descent result to pass to the perfect closure of $K$.

\begin{lemma}
\label{perfdesc}
Assume that $k$ is algebraically closed. Let $X$
be a smooth and purely $d$-dimensional scheme over $k$, let
$\mathcal{F}\in D_{\rm ctf}(X,\Lambda)$, and let
$K^{\operatorname{perf}}$ be the perfect closure of $K$.
Then we have
\begin{equation}\label{eq:CC-perfect-pullback}
 SS
   (\widetilde{\mathcal F}_{\eta^{\operatorname{perf}}})
   =SS(\mathcal F)\times_k\eta^{\operatorname{perf}},
 \qquad
 CC
   (\widetilde{\mathcal F}_{\eta^{\operatorname{perf}}})
   =CC(\mathcal F)\times_k\eta^{\operatorname{perf}},
\end{equation}
where the second term is the flat pullback of cycle. 
\end{lemma}

\begin{proof}
Under the canonical identification
\begin{equation}\label{eq:display-1559}
 T^{\ast}(\mathcal X_{\eta^{\operatorname{perf}}}/K^{\operatorname{perf}})
 \simeq
 T^{\ast}(X/k)\times_k\eta^{\operatorname{perf}},
\end{equation}
the complex $\widetilde{\mathcal F}_{\eta^{\operatorname{perf}}}$ is the
scalar extension of $\mathcal F$. Hence by
\cite[Theorem~1.4(iii)]{Bei16} we have
\begin{equation}\label{eq:display-1567}
 SS(\widetilde{\mathcal F}_{\eta^{\operatorname{perf}}})
 =SS(\mathcal F)\times_k\eta^{\operatorname{perf}}.
\end{equation}
Put
\begin{equation}\label{eq:display-1573}
 C=SS(\mathcal F)=\bigcup_\alpha C_\alpha,
 \qquad
 CC(\mathcal F)=\sum_\alpha m_\alpha[C_\alpha],
\end{equation}
where the $C_\alpha$ are all the irreducible components of $C$, allowing
$m_\alpha=0$, and set
\begin{equation}\label{eq:display-1580}
 C_{\operatorname{perf}}
 =C\times_k\eta^{\operatorname{perf}},
 \qquad
 B=CC(\mathcal F)\times_k\eta^{\operatorname{perf}}.
\end{equation}
Every irreducible component of $C_{\operatorname{perf}}$ has dimension $d$,
and $B$ is supported on $C_{\operatorname{perf}}$. It suffices to verify the
Milnor formula for $B$.

Let
\begin{equation}\label{eq:display-1591}
 j:U\longrightarrow \mathcal X_{\eta^{\operatorname{perf}}},
 \qquad
 f:U\longrightarrow Y
\end{equation}
be a test pair with an isolated characteristic point, i.e., $j$ is \'etale, $Y$ is a
smooth curve over $K^{\operatorname{perf}}$, and $u_0\in U$ is a closed
point, with a geometric point $\bar u$ above it, which is at most an isolated
$j^*C_{\operatorname{perf}}$-characteristic point of $f$.

If $f$ is $j^*C_{\operatorname{perf}}$-transversal at $u_0$, then we may shrink $U$
around $u_0$. The local intersection number is zero and
\cite[Lemma~4.2(5)]{Sai17a} gives
\begin{equation}\label{eq:display-1604}
 R\Phi_{\bar u}
 \bigl(j^*\widetilde{\mathcal F}_{\eta^{\operatorname{perf}}},f\bigr)=0.
\end{equation}
Thus the Milnor formula holds in this case.

Assume henceforth that $u_0$ is isolated characteristic point. 
In this case, we may assume $Y=\A^1_{K^{\operatorname{perf}}}$ 
by shrinking around $u_0$ and passing to an \'etale cover.
Write
\begin{equation}\label{eq:display-1614}
 df:U\longrightarrow T^{\ast}(U/K^{\operatorname{perf}})
\end{equation}
for the section defined by the relative differential of $f$. Since $j$ is \'etale, we may regard 
$j^*C_{\operatorname{perf}}$
as a closed conical subset of
$T^{\ast}(U/K^{\operatorname{perf}})$, endowed with its reduced scheme structure.
After shrinking $U$ once more, assume that the scheme-theoretic
characteristic locus
\begin{equation}\label{eq:display-1626}
 Z_{\operatorname{perf}}
 :=U\times_{df,\,T^{\ast}(U/K^{\operatorname{perf}})}
   j^*C_{\operatorname{perf}}
\end{equation}
is finite over $K^{\operatorname{perf}}$.

The field $K$ is algebraic over $k(u)$, since it is the fraction field of the
strict henselization of $k[u]_{(u)}$, and
$K^{\operatorname{perf}}/K$ is purely inseparable. Hence
$K^{\operatorname{perf}}/k(u)$ is algebraic.
Since all data in the test pair are of
finite presentation, they descend to a finitely generated integral
$k[u]$-subalgebra $A_0\subset K^{\operatorname{perf}}$.
The extension $\operatorname{Frac}(A_0)/k(u)$ is finite. Since $k$ is
perfect, we may find a localization of $A_0$, denoted by $A$, such that
$ A\subset K^{\operatorname{perf}}$
is a finitely generated smooth integral $k$-algebra containing $k[u]$ and
carrying the descended data. Put $ M=\Spec A, F_0=\operatorname{Frac}(A).$
Thus $A$ is of finite type over $k$, and
$K^{\operatorname{perf}}/F_0$ is algebraic. After enlarging the model if
necessary, we obtain an \'etale morphism and an $M$-morphism
\begin{equation}\label{eq:display-1653}
 j_A:U_A\longrightarrow X\times_kM,
 \qquad
 f_A:U_A\longrightarrow\A^1_M
\end{equation}
whose base change to $\eta^{\operatorname{perf}}$ is the given test. Put
\begin{equation}\label{eq:display-1659}
 h=\operatorname{pr}_X\circ j_A:U_A\rightarrow X,
 \quad
 \pi=\operatorname{pr}_M\circ j_A:U_A\longrightarrow M,
 \quad
 \mathcal F_A=j_A^*\operatorname{pr}_X^*\mathcal F,
\end{equation}
and, using the cotangent-bundle isomorphism induced by $j_A$, put
\begin{equation}\label{eq:display-1667}
 C_A=U_A\times_X C\subset T^{\ast}(U_A/M),
\end{equation}
\begin{equation}\label{eq:display-1670}
 B_A
 =j_A^*\bigl(CC(\mathcal F)\times_kM\bigr)
 =\sum_\alpha m_\alpha[U_A\times_XC_\alpha].
\end{equation}
Let $ df_A:U_A\longrightarrow T^{\ast}(U_A/M)$
be the section defined by the relative differential of $f_A$, and define
\begin{equation}\label{eq:display-1680}
 Z_A
 =U_A\times_{df_A,\,T^{\ast}(U_A/M)}C_A.
\end{equation}
We have $(Z_A)_{\eta^{\operatorname{perf}}}=Z_{\operatorname{perf}}$.

The generic fiber $(Z_A)_{F_0}$ is finite, since it becomes the finite scheme
$Z_{\operatorname{perf}}$ after the faithfully flat extension
$F_0\subset K^{\operatorname{perf}}$. The open quasi-finite locus of
$Z_A\to M$ therefore contains the generic fiber. Its complement has only
finitely many irreducible components, none of which dominates $M$. After
removing the closures of their images from $M$, we may assume that
$ Z_A\longrightarrow M$
is quasi-finite.


Since $j_A=(h,\pi)$ is \'etale, the morphism $h$ is smooth. Hence $h$
is $C$-transversal and
$SS(\mathcal F_A)=h^\circ C$
by \cite[Lemma~3.4(1) and Corollary~4.5(2)]{Sai17a}. Under the cotangent
decomposition induced by $j_A$, this equality becomes
\begin{equation}\label{eq:display-1734}
 T^{\ast}(U_A/k)
 \simeq T^{\ast}(U_A/M)\oplus\pi^*T^{\ast}(M/k),
 \qquad
 h^\circ C=C_A\oplus 0.
\end{equation}
For the $M$-morphism $f_A$, the composition of its cotangent map 
\begin{equation}\label{eq:auto-067}
df_A: T^*(\A^1_M/k)\times_{\A^1_M} U_A
\rightarrow T^*(U_A/k)
\end{equation}
with the projection
$T^{\ast}(U_A/k)\longrightarrow T^{\ast}(U_A/M)$
sends $(a\,dt,\beta)$ to $a\,df_A$, while covectors pulled back from
$M$ map identically to the second summand $\pi^*T^{\ast}(M/k)$.
Consequently, by the definition of transversality and the conicality of
$C_A$, the non-$C$-transversal locus of $(h,f_A)$ is precisely the
underlying set of
$ (df_A)^{-1}(C_A)=Z_A.$
Thus $(h,f_A)$ is $C$-transversal on $U_A\setminus Z_A$. Since the
cotangent map of $\pi$ has image in the second summand $\pi^*T^{\ast}(M/k)$, which meets
$C_A\oplus 0$ only in the zero-section, the pair $(h,\pi)$ is
$C$-transversal everywhere.

For a geometric point $\bar z$ of $Z_A$, with image $\bar s$ in $M$, define
\begin{equation}\label{eq:display-1765}
\begin{aligned}
I(\bar z)
 &=\bigl(B_{A,\bar s},df_{A,\bar s}\bigr)
   _{T^{\ast}(U_{A,\bar s}/\bar s),\bar z},\\
 \phi(\bar z)
 &=-\operatorname{dimtot}\,
   R\Phi_{\bar z}(\mathcal F_{A,\bar s},f_{A,\bar s}).
\end{aligned}
\end{equation}
For every $\alpha$, set
\begin{equation}\label{eq:display-1775}
 \mathcal A_{\alpha,A}
 =\mathcal O_{U_A\times_XC_\alpha}
  \mathbin{\otimes^{\mathbf L}}_{\mathcal O_{T^{\ast}(U_A/M)}，df_A}
  \mathcal O_{U_A},
\end{equation}
where the last module structure is defined by the section $df_A$. Since
$M$ is smooth over $k$ and $U_A$ is smooth over $M$, the scheme
$T^{\ast}(U_A/M)$ is regular. Hence
$\mathcal O_{U_A\times_XC_\alpha}$ is perfect on it, and
$\mathcal A_{\alpha,A}$ is perfect on $U_A$. Since $U_A\to M$ is smooth,
$\mathcal A_{\alpha,A}$ has finite Tor-dimension over $M$.
Moreover,
\begin{equation}\label{eq:display-1788}
 U_A\times_XC_\alpha
 =U_A\times_{X\times_kM}(C_\alpha\times_kM)
\end{equation}
is \'etale over $C_\alpha\times_kM$, and therefore is flat over $M$; the
scheme $U_A$ is also flat over $M$. By derived base change we have
\begin{equation}\label{eq:auto-068}
\left(\mathcal{A}_{\alpha,A}\right)_{\overline{s}}
=
Li^*_{\overline{s}}\mathcal{A}_{\alpha,A}
\simeq
\mathcal O_{U_{A,\overline{s}}\times_X C_{\alpha}}
\mathbin{\otimes^{\mathbf L}}_{\mathcal O_{T^*(U_{A,\overline{s}}/\overline{s})},df_{A,\overline{s}}}
\mathcal O_{U_{A,\overline{s}}}
\end{equation}
where
$\overline{s}\rightarrow M$
is any geometric point.
Thus
$\left(\mathcal{A}_{\alpha,A}\right)_{\overline{s}}$
is
the derived
local intersection of
$U_{A,\overline{s}}\times_XC_\alpha$
and $df_{A,\overline{s}}(U_{A,\overline{s}})$
in $T^{\ast}(U_{A,\bar s}/\bar s)$. Its local Euler
characteristic is the corresponding isolated intersection multiplicity in the definition of the function $I$.
The cohomology of $\mathcal A_{\alpha,A}$ is supported on
\begin{equation}\label{eq:display-1798}
 (U_A\times_XC_\alpha)\cap df_A(U_A)\subset Z_A,
\end{equation}
hence on a closed subscheme quasi-finite over $M$. Therefore
\cite[Lemma~2.3 and the proof of Proposition~5.8]{Sai17a}, followed by
summing with the coefficients $m_\alpha$, shows that $I$ is constructible
and flat over $M$ in the sense of \cite[Definition~2.1]{Sai17a}.

The transversality of $(h,\pi)$ and of $(h,f_A)$ outside $Z_A$, together with
\cite[Lemma~4.2(5)]{Sai17a}, shows that $\pi$ is universally locally
acyclic relatively to $\mathcal F_A$ and that $f_A$ is universally locally
acyclic relatively to $\mathcal F_A$ on $U_A\setminus Z_A$. By
\cite[Proposition~2.16]{Sai17a} we have $\phi$ is also constructible and
flat over $M$ since
$\A^1_M\to M$ is a smooth relative curve and $Z_A\to M$ is quasi-finite.

Let $t$ be a closed point of $M$. Since $A$ is of finite type over the
algebraically closed field $k$, one has $\kappa(t)=k$, and
$ h_t:U_{A,t}\longrightarrow X$
is \'etale. By \cite[Lemma~5.11(2)]{Sai17a}, we have
\begin{equation}\label{eq:display-1820}
 B_{A,t}
 =h_t^*CC(\mathcal F)
 =CC(h_t^*\mathcal F)
 =CC(\mathcal F_{A,t}).
\end{equation}
The fiber $Z_{A,t}$ is finite. Saito's Milnor formula
\cite[Theorem~5.9]{Sai17a} therefore gives 
\begin{equation}
    \label{eq:iphicls}
    I(\bar z)=\phi(\bar z)
\end{equation}
for every geometric point $\bar z$ above every closed point $t\in M$.

Notice that $I-\phi$ is also a constructible function.
On every irreducible
component $Z'$ of $Z_A$, constructibility gives a dense open subset $V$ on which
$I-\phi$ equals its value at the geometric generic point of $Z'$. Since $Z'$ is
of finite type over the algebraically closed field $k$, it is Jacobson, and
$V$ contains a closed $k$-point $z$; its image in $M$, $z_M$, is also
closed. Hence regard \(\bar z\) as a geometric point of \(Z_A\) lying above $z_M$, by \eqref{eq:iphicls} we have $I(\overline{z})=\phi (\overline{z})$.
Thus $I-\phi$ vanishes at the generic point of every irreducible component of
$Z_A$. Saito \cite[Lemma~2.2(3)]{Sai17a} now gives
$ I-\phi=0$
on all geometric points of $Z_A$.

Choose an algebraic closure $\Omega$ of $K^{\operatorname{perf}}$ together
with an embedding $\kappa(u_0)\hookrightarrow\Omega$
denoted by $\bar u$. Since $K^{\operatorname{perf}}/F_0$ is algebraic,
$\Omega$ is also an algebraic closure of $F_0$, so $\bar u$ defines a
geometric point of $Z_A$ above the geometric generic point of $M$.
Moreover, $\kappa(u_0)/K^{\operatorname{perf}}$ is finite separable because
$K^{\operatorname{perf}}$ is perfect. Flat base change therefore preserves
the local intersection multiplicity at the chosen geometric point, hence we have
\begin{equation}\label{eq:display-1894}
 I(\bar u)
 =\bigl(j^*B,df\bigr)_
   {T^{\ast}(U/K^{\operatorname{perf}}),u_0}.
\end{equation}
Since $Z_A\to M$ is quasi-finite, the induced morphism
$Z_A\to\A^1_M$ is also quasi-finite. By the universal local
acyclicity of $f_A$ outside $Z_A$ and \cite[Proposition~2.8(1)]{Sai17a},
we may identify the vanishing-cycle complex defining
$\phi(\bar u)$ with $R\Phi_{\overline{u}}(j^*\widetilde{\mathcal F}_{\eta^{\operatorname{perf}}},f)$.
Consequently,
\begin{equation}\label{eq:display-1908}
 \phi(\bar u)
 =-\operatorname{dimtot}\,
  R\Phi_{\bar u}
  \bigl(j^*\widetilde{\mathcal F}_{\eta^{\operatorname{perf}}},f\bigr).
\end{equation}
Evaluating $I-\phi=0$ at $\bar u$ yields
\begin{equation}\label{eq:display-1919}
 \bigl(j^*B,df\bigr)_{T^{\ast}(U/K^{\operatorname{perf}}),u_0}
 =-\operatorname{dimtot}\,
  R\Phi_{\bar u}
  \bigl(j^*\widetilde{\mathcal F}_{\eta^{\operatorname{perf}}},f\bigr).
\end{equation}
Thus $B$ satisfies the Milnor formula for every at most isolated
$C_{\operatorname{perf}}$-characteristic test. Since
$ C_{\operatorname{perf}}
 =SS(\widetilde{\mathcal F}_{\eta^{\operatorname{perf}}})$,
by the uniqueness in \cite[Theorem~5.9]{Sai17a} we have
\begin{equation}\label{eq:display-1932}
 CC(\widetilde{\mathcal F}_{\eta^{\operatorname{perf}}})
 =CC(\mathcal F)\times_k\eta^{\operatorname{perf}}.
\end{equation}
This proves \eqref{eq:CC-perfect-pullback}.

\end{proof}

\subsection{}
We now verify the local-acyclicity condition on the ordinary generic
fiber.  By
\eqref{eq:CC-perfect-pullback},
\begin{equation}\label{eq:display-1947}
 SS(\mathcal F)\times_k\eta^{\operatorname{perf}}=SS(\widetilde{\mathcal F}_{\eta^{\operatorname{perf}}}).
\end{equation}
Since $\phi_{\eta^{\operatorname{perf}}}$ is $SS(\mathcal F)_{\eta^{\operatorname{perf}}}$-transversal outside $\mathcal W_{\eta^{\operatorname{perf}}}$,
\cite[Lemma~4.2(5)]{Sai17a} implies that $\phi_{\eta^{\operatorname{perf}}}$ is universally
locally acyclic relatively to $\widetilde{\mathcal F}_{\eta^{\operatorname{perf}}}$ outside
$\mathcal W_{\eta^{\operatorname{perf}}}$.
Notice that the morphism $\eta^{\operatorname{perf}}=\Spec K^{\operatorname{perf}}\to\Spec K=\eta$ is a universal homeomorphism，
The induced equivalences of \'etale topoi
\cite[Tag~03SI]{Stacks}, together with base change
identify the canonical morphisms in
Lemma~\ref{lem:Trans=ULA}(4) before and after this extension. 
Hence universal local acyclicity descends by the equivalence
of (2) and (4) in Lemma~\ref{lem:Trans=ULA}. Therefore, we have
\begin{equation}\label{eq:display-1958}
 \phi_\eta:\mathcal X_\eta\longrightarrow\mathcal Y_\eta
\end{equation}
is universally locally acyclic relatively to
$\widetilde{\mathcal F}_\eta$ outside $\mathcal W_\eta$.
By a slight abuse of notation, we denote both the specialization map on $H^0(\mathcal X_\eta,\mathcal K_{\mathcal X_{\eta}/\eta})$ and $H^0(\mathcal X_\eta,\mathcal K_{\mathcal X_\eta/\mathcal Y_\eta/\eta})$ by $\spet$.
The universal local acyclicity of
$\operatorname{pr}_S:\mathcal X\to S$ gives
\begin{equation}\label{eq:display-1966}
 R\Psi_{\operatorname{pr}_S}(\widetilde{\mathcal F}_\eta)
 \simeq \widetilde{\mathcal F}_s
 \simeq \mathcal F ,
\end{equation}
therefore Proposition~\ref{NAprop}(6) gives
\begin{equation}\label{eq:nasp}
 \spet
 \cC^{\mathcal W_\eta}_{
   \mathcal X_\eta/\mathcal Y_\eta/\eta}
   (\widetilde{\mathcal F}_\eta)
 =
 \cC^{Z}_{X/\A^1_k/k}(\mathcal F).
\end{equation}

\subsection{}
Recall that we have the section $\sigma=d\phi=df+u\,dg$.
Combining Lemma~\ref{perfdesc} with flat base change for the refined intersection, we have
\begin{equation}\label{eq:perfsecbc}
  \sigma_\eta^!\bigl(CC(\mathcal F)\times_k\eta\bigr)\times_\eta \eta^{\operatorname{perf}}
 =\sigma_{\eta^{\operatorname{perf}}}^!
   \bigl(CC(\mathcal F)\times_k\eta^{\operatorname{perf}}\bigr)
=\sigma_{\eta^{\operatorname{perf}}}^!
CC(\widetilde{\mathcal F}_{\eta^{\operatorname{perf}}}).
\end{equation}

\begin{lemma}
    \label{genperfmil}

    With the preceding notation, we have 
    \begin{equation}\label{eq:perfect-equality-manuscript}
 \cC^{\mathcal W_{\eta^{\operatorname{perf}}}}_{
 \mathcal X_{\eta^{\operatorname{perf}}}/
 \mathcal Y_{\eta^{\operatorname{perf}}}/\eta^{\operatorname{perf}}}
 (\widetilde{\mathcal F}_{\eta^{\operatorname{perf}}})
 =\widetilde{{\rm cl}}\!\left(
 \sigma_{\eta^{\operatorname{perf}}}^!
   \bigl(CC(\mathcal F)\times_k\eta^{\operatorname{perf}}\bigr)
 \right).
\end{equation}
\end{lemma}

\begin{proof}
   Recall that under our construction, $\mathcal{W}_{\eta^{\operatorname{perf}}}$ is finite over $\eta^{\operatorname{perf}}$. For each closed point
$x\in|\mathcal W_{\eta^{\operatorname{perf}}}|$, choose an open neighbourhood $U_x$
containing no other point of $|\mathcal W_{\eta^{\operatorname{perf}}}|$. The \'etale
locality and the cohomological Milnor formula of non-acyclicity class \cite[Proposition~4.6 and Theorem~6.1]{YZ22} give
\begin{equation}\label{eq:YZ-point-value}
 C^{\{x\}}_{U_x/\mathcal Y_{\eta^{\operatorname{perf}}}/\eta^{\operatorname{perf}}}
   (\widetilde{\mathcal F}_{\eta^{\operatorname{perf}}}|_{U_x})
 =-\operatorname{dimtot}
   R\Phi_{\bar x}(\widetilde{\mathcal F}_{\eta^{\operatorname{perf}}},
                  \phi_{\eta^{\operatorname{perf}}})
 \quad\text{in }H_x^0(U_x,\mathcal K_{U_x/\eta^{\operatorname{perf}}})\simeq\Lambda.
\end{equation}
Saito's isolated characteristic-cycle formula
\cite[Theorem~5.9]{Sai17a} says, for the same point,
\begin{equation}\label{eq:Saito-point-value}
\begin{aligned}
&\operatorname{mult}_x\!\left(
  \sigma_{\eta^{\operatorname{perf}}}^!
   CC(\widetilde{\mathcal F}_{\eta^{\operatorname{perf}}})\right)\\
 &\qquad=
 \bigl(CC(\widetilde{\mathcal F}_{\eta^{\operatorname{perf}}}),
       d\phi_{\eta^{\operatorname{perf}}}\bigr)_{
          T^{\ast}(\mathcal X_{\eta^{\operatorname{perf}}}/\eta^{\operatorname{perf}}),x}
 =-\operatorname{dimtot}
   R\Phi_{\bar x}(\widetilde{\mathcal F}_{\eta^{\operatorname{perf}}},
                  \phi_{\eta^{\operatorname{perf}}})
 \quad\text{in }\mathbf Z.
\end{aligned}
\end{equation}
Since supported cohomology on a finite set decomposes as the direct sum of its
point-supported groups, summing over $x$ gives
\begin{equation}\label{eq:display-2040}
 \cC^{\mathcal W_{\eta^{\operatorname{perf}}}}_{
 \mathcal X_{\eta^{\operatorname{perf}}}/
 \mathcal Y_{\eta^{\operatorname{perf}}}/\eta^{\operatorname{perf}}}
 (\widetilde{\mathcal F}_{\eta^{\operatorname{perf}}})
 =\widetilde{{\rm cl}}
 \bigl(\sigma_{\eta^{\operatorname{perf}}}^!
CC(\widetilde{\mathcal F}_{\eta^{\operatorname{perf}}})
\bigr).
\end{equation}
Applying \eqref{eq:perfsecbc} we get the desired result.
\end{proof}

\begin{lemma}
\label{nacohgrpdesc}
Let $q_X:\mathcal X_{\eta^{\operatorname{perf}}}\to\mathcal X_\eta$ and
$q_Y:\mathcal Y_{\eta^{\operatorname{perf}}}\to\mathcal Y_\eta$ be the projections.  Then there is a canonical isomorphism
\begin{equation}\label{eq:display-2059}
 q_X^*\mathcal K_{\mathcal X_\eta/\mathcal Y_\eta/\eta}
 \xrightarrow{\sim}
 \mathcal K_{\mathcal X_{\eta^{\operatorname{perf}}}/\mathcal Y_{\eta^{\operatorname{perf}}}/\eta^{\operatorname{perf}}}.
\end{equation}
Moreover, we have
\begin{equation}\label{eq:display-2065}
 q_X^*:H^0_{\mathcal W_\eta}
   (\mathcal X_\eta,\mathcal K_{\mathcal X_\eta/\mathcal Y_\eta/\eta})
 \xrightarrow{\sim}
 H^0_{\mathcal W_{\eta^{\operatorname{perf}}}}
   (\mathcal X_{\eta^{\operatorname{perf}}},\mathcal K_{\mathcal X_{\eta^{\operatorname{perf}}}/\mathcal Y_{\eta^{\operatorname{perf}}}/\eta^{\operatorname{perf}}}).
\end{equation}
\end{lemma}

\begin{proof}
The maps $q_X$ and $q_Y$ are universal homeomorphisms, hence $q_X^*$ and $q_Y^*$ are equivalences on the relevant
derived categories of constructible sheaves by \cite[Expos\'e~VIII, Th\'eor\`eme~1.1]{SGA4}.  For the
Cartesian square defined by
$\phi_\eta:\mathcal X_\eta\to\mathcal Y_\eta$, base change gives
\begin{equation}\label{eq:display-2079}
 q_Y^*R\phi_{\eta!}\xrightarrow{\sim}R\phi_{\eta^{\operatorname{perf}}!}q_X^*,
\end{equation}
by \cite[Tag~0F7L]{Stacks}.  Moreover, the two equivalences yield a canonical base change isomorphism
\begin{equation}\label{eq:display-2083}
 q_X^*R\phi_\eta^!\xrightarrow{\sim}R\phi_{\eta^{\operatorname{perf}}}^!q_Y^*.
\end{equation}
Hence by definition we have 
\begin{equation}\label{eq:display-2087}
 q_X^*\mathcal K_{\mathcal X_\eta/\eta}
 \xrightarrow{\sim}
 \mathcal K_{\mathcal X_{\eta^{\operatorname{perf}}}/\mathcal \eta^{\operatorname{perf}}},\qquad 
 q_X^*\mathcal K_{\mathcal X_\eta/\mathcal Y_\eta}
 \xrightarrow{\sim}
 \mathcal K_{\mathcal X_{\eta^{\operatorname{perf}}}/\mathcal Y_{\eta^{\operatorname{perf}}}},
\end{equation}
and the consequence follows from the distinguished triangle \eqref{eq:distriK}.
\end{proof}

\begin{lemma}
    \label{genmil}
    Retain the preceding notation. We have 
\begin{equation}\label{eq:perfdesc}
 \cC^{\mathcal W_\eta}_{
 \mathcal X_\eta/\mathcal Y_\eta/\eta}(\widetilde{\mathcal F}_\eta)
 =\widetilde{{\rm cl}}\!\left(
  \sigma_\eta^!\bigl(CC(\mathcal F)\times_k\eta\bigr)
 \right).
\end{equation}

\end{lemma}
\begin{proof}
Let $q_X$ and $q_Y$ be as before.  The distinguished triangle \eqref{eq:distriK}
 gives 
 \begin{equation}\label{eq:display-2113}
\rho_s:H^0(X,\mathcal K_{ X/k})\rightarrow H^0(X,\mathcal K_{X/\A^1_k/k}),\qquad
\rho_\eta:H^0(\mathcal X_\eta,\mathcal K_{\mathcal X_{\eta}/\eta})
\rightarrow H^0(\mathcal X_\eta,\mathcal K_{\mathcal X_\eta/\mathcal Y_\eta/\eta}),
 \end{equation}
and similarly we define $\rho_{\eta^{\operatorname{perf}}}$.
  Thus
\begin{equation}\label{eq:display-2120}
 \widetilde{{\rm cl}}_\eta
 =\rho_\eta\circ{\rm cl}_\eta,
 \qquad
 \widetilde{{\rm cl}}_{\eta^{\operatorname{perf}}}
 =\rho_{\eta^{\operatorname{perf}}}
    \circ{\rm cl}_{\eta^{\operatorname{perf}}}.
\end{equation}
We compare the pullbacks of the two sides of
\eqref{eq:perfdesc}.  Firstly, the pullback formula for the
non-acyclicity class \eqref{eq:pbCNA} gives
\begin{equation}\label{eq:CNA-perfect-pullback}
 q_X^*
 \cC^{\mathcal W_\eta}_{
   \mathcal X_\eta/\mathcal Y_\eta/\eta}
   (\widetilde{\mathcal F}_\eta)
 =
 \cC^{\mathcal W_{\eta^{\operatorname{perf}}}}_{
   \mathcal X_{\eta^{\operatorname{perf}}}/
   \mathcal Y_{\eta^{\operatorname{perf}}}/
   \eta^{\operatorname{perf}}}
   (\widetilde{\mathcal F}_{\eta^{\operatorname{perf}}}).
\end{equation}
On the cycle side, compatibility of the supported cycle class
map with scalar extension
\cite[Cycle, Lemme~2.3.7]{SGA4h} gives
\begin{equation}\label{eq:display-2148}
 q_X^*{\rm cl}_\eta(\sigma_\eta^!\bigl(CC(\mathcal F)\times_k\eta\bigr))
 =
 {\rm cl}_{\eta^{\operatorname{perf}}}
 \bigl(
   \sigma_\eta^!\bigl(CC(\mathcal F)\times_k\eta\bigr)\times_\eta\eta^{\operatorname{perf}}
 \bigr).
\end{equation}
Moreover, the naturality of morphisms involved gives the commutative square
\begin{equation}\label{eq:display-2157}
\begin{CD}
 H^0_{\mathcal W_\eta}
   (\mathcal X_\eta,\mathcal K_{\mathcal X_\eta/\eta})
 @>{\rho_\eta}>>
 H^0_{\mathcal W_\eta}
   (\mathcal X_\eta,
    \mathcal K_{\mathcal X_\eta/\mathcal Y_\eta/\eta})\\
 @V{q_X^*}VV @VV{q_X^*}V\\
 H^0_{\mathcal W_{\eta^{\operatorname{perf}}}}
   (\mathcal X_{\eta^{\operatorname{perf}}},
    \mathcal K_{\mathcal X_{\eta^{\operatorname{perf}}}/
      \eta^{\operatorname{perf}}})
 @>{\rho_{\eta^{\operatorname{perf}}}}>>
 H^0_{\mathcal W_{\eta^{\operatorname{perf}}}}
   (\mathcal X_{\eta^{\operatorname{perf}}},
    \mathcal K_{\mathcal X_{\eta^{\operatorname{perf}}}/
      \mathcal Y_{\eta^{\operatorname{perf}}}/
      \eta^{\operatorname{perf}}}).
\end{CD}
\end{equation}
Consequently, using \eqref{eq:perfsecbc}, we have
\begin{equation}\label{eq:tilde-cl-perfect-pullback}
\begin{aligned}
q_X^*\widetilde{{\rm cl}}_\eta(\sigma_\eta^!\bigl(CC(\mathcal F)\times_k\eta\bigr))
 &=
 \widetilde{{\rm cl}}_{\eta^{\operatorname{perf}}}
 \bigl(
   \sigma_\eta^!
 \bigl(CC(\mathcal F)\times_k\eta\bigr)
 \times_\eta\eta^{\operatorname{perf}}
 \bigr)\\
 &=
 \widetilde{{\rm cl}}_{\eta^{\operatorname{perf}}}
   \bigl(\sigma_{\eta^{\operatorname{perf}}}^!
 \bigl(CC(\mathcal F)\times_k\eta^{\operatorname{perf}}\bigr)\bigr).
\end{aligned}
\end{equation}
By \eqref{eq:perfect-equality-manuscript},
the right-hand sides of
\eqref{eq:CNA-perfect-pullback} and
\eqref{eq:tilde-cl-perfect-pullback} are equal,  hence
\begin{equation}\label{eq:display-2199}
 q_X^*
 \cC^{\mathcal W_\eta}_{
   \mathcal X_\eta/\mathcal Y_\eta/\eta}
   (\widetilde{\mathcal F}_\eta)
 =
 q_X^*\widetilde{{\rm cl}}_\eta
 \bigl(\sigma_\eta^!
 \bigl(CC(\mathcal F)\times_k\eta\bigr)
 \bigr).
\end{equation}
Finally, by the previous lemma~\ref{nacohgrpdesc}, the equality descends uniquely to
$\eta$, and we obtain
\begin{equation}\label{eq:display-2213}
 \cC^{\mathcal W_\eta}_{
   \mathcal X_\eta/\mathcal Y_\eta/\eta}
   (\widetilde{\mathcal F}_\eta)
 =
 \widetilde{{\rm cl}}_\eta\!\left(
   \sigma_\eta^!
   \bigl(CC(\mathcal F)\times_k\eta\bigr)
 \right)
\end{equation}
which is \eqref{eq:perfdesc}.
\end{proof}
 
\subsection{\normalfont\itshape Proof of Theorem~\ref{thmqafmil}}
\pushQED{\qed}

Recall that
 \begin{equation}\label{eq:display-2230}
\rho_s:H^0(X,\mathcal K_{ X/k})\rightarrow H^0(X,\mathcal K_{X/\A^1_k/k}),\qquad
\rho_\eta:H^0(\mathcal X_\eta,\mathcal K_{\mathcal X_{\eta}/\eta})
\rightarrow H^0(\mathcal X_\eta,\mathcal K_{\mathcal X_\eta/\mathcal Y_\eta/\eta})
 \end{equation}
 and by definition we have $\widetilde{{\rm cl}}_s=\rho_s\circ{\rm cl}_s$, $\widetilde{{\rm cl}}_\eta=\rho_\eta\circ{\rm cl}_\eta$. 
 By \cite[(4.47)]{YZ22}, we have 
 \begin{equation}\label{eq:display-2237}
 \spet \circ \rho_\eta=\rho_s\circ \spet.
 \end{equation}
Notice that every component of the cycle 
$CC(\mathcal F)\times_k[S]\in
 {\rm CH}_{d+1}\bigl(SS(\mathcal F)\times_k S/S\bigr)$ 
 is flat over $S$, Corollary~\ref{clspcomm} gives
\begin{equation}\label{eq:CC-specialization}
 \spet{\rm cl}_\eta\!\left(
   \sigma_\eta^!\bigl(CC(\mathcal F)\times_k\eta\bigr)
 \right)
 =
 {\rm cl}_s\!\left((df)^!CC(\mathcal F)\right).
\end{equation}
Combining \eqref{eq:nasp}, \eqref{eq:perfdesc} and \eqref{eq:CC-specialization}, we have
\begin{equation}\label{eq:display-2257}
\begin{aligned}
\cC^{Z}_{X/\A^1/k}(\mathcal F)
 &=\spet\cC^{\mathcal W_\eta}_{
   \mathcal X_\eta/\mathcal Y_\eta/\eta}(\widetilde{\mathcal F}_\eta)\\
 &=\spet\rho_\eta{\rm cl}\!\left(
    \sigma_\eta^!\bigl(CC(\mathcal F)\times_k\eta\bigr)\right)\\
 &=\rho_s\spet{\rm cl}\!\left(
    \sigma_\eta^!\bigl(CC(\mathcal F)\times_k\eta\bigr)\right)\\
 &=\rho_s{\rm cl}\bigl((df)^!CC(\mathcal F)\bigr).
\end{aligned}
\end{equation}
This is Theorem~\ref{thmqafmil}.

\popQED


\begin{thebibliography}{20}

	
	\bibitem{AS07}
	A. Abbes and T. Saito,
	\emph{The characteristic class and ramification of an $\ell$-adic \'etale sheaf}, Inventiones mathematicae, 168 (2007): 567-612.
	
	\bibitem{Abe21}
	T. Abe, \emph{Ramification theory from homotopical point of view, I},
	arXiv:2206.02401v2 (2026), \url{https://arxiv.org/abs/2206.02401v2}.
	
	
	
	\bibitem{SGA4}
	M. Artin, A. Grothendieck, and J.-L. Verdier,
	\emph{Th\'eorie des topos et cohomologie \'etale des sch\'emas},
	S\'eminaire de G\'eom\'etrie Alg\'ebrique du Bois-Marie 1963--1964 (SGA~4),
	avec la collaboration de N. Bourbaki, P. Deligne, and B. Saint-Donat,
	Lecture Notes in Math., vols. 269, 270, and 305, Springer, Berlin, 1972--1973.
	
	
	\bibitem{Bei16}
	A. Beilinson, \emph{Constructible sheaves are holonomic},
	Selecta Math. (N.S.) 22 (2016), 1797--1819,
	\href{https://doi.org/10.1007/s00029-016-0260-z}{doi:10.1007/s00029-016-0260-z}.
	
	
	
	
	
	\bibitem{Del73}
	P. Deligne, \emph{La formule de dualit\'e globale}, Expos\'e XVIII,
	in \emph{Th\'eorie des topos et cohomologie \'etale des sch\'emas} (SGA~4),
	Tome 3, Lecture Notes in Math., vol. 305, Springer, Berlin, 1973,
	pp. 481--587.
	
	
	\bibitem{Del72}
	P. Deligne, \emph{La formule de Milnor}, Expos\'e XVI,
	in \emph{Groupes de monodromie en g\'eom\'etrie alg\'ebrique} (SGA~7 II),
	Lecture Notes in Math., vol. 340, Springer, Berlin, 1973, pp. 197--211.
	
	\bibitem{SGA4h}
	P. Deligne, \emph{Cohomologie \'etale} (SGA~$4\frac{1}{2}$),
	avec la collaboration de J.-F. Boutot, A. Grothendieck, L. Illusie, and J.-L. Verdier,
	Lecture Notes in Math., vol. 569, Springer, Berlin, 1977.
	
	\bibitem{Del11}
	P. Deligne, \emph{Notes sur Euler--Poincar\'e: brouillon project},
	unpublished notes, February 8, 2011.
	
	\bibitem{Ful98}
	W. Fulton, \emph{Intersection Theory}, 2nd ed.,
	Ergebnisse der Mathematik und ihrer Grenzgebiete. 3. Folge,
	vol. 2, Springer-Verlag, Berlin, 1998.

%
%
%
	
%
%
%
%
	
	
	
	
	\bibitem{LZ19}
	Q. Lu and W. Zheng, \emph{Duality and nearby cycles over general bases},
	Duke Math. J. 168 (2019), 3135--3213,
	\href{https://doi.org/10.1215/00127094-2019-0057}{doi:10.1215/00127094-2019-0057}.
	
	\bibitem{LZ22}
	Q. Lu and W. Zheng, \emph{Categorical traces and a relative Lefschetz--Verdier formula},
	Forum Math. Sigma 10 (2022), Paper No. e10, 24 pp.,
	\href{https://doi.org/10.1017/fms.2022.2}{doi:10.1017/fms.2022.2}.
	
%
%
%

	\bibitem{Rio14}
	J. Riou, \emph{Classes de Chern, morphismes de Gysin, puret\'e absolue},
	in \emph{Travaux de Gabber sur l'uniformisation locale et la cohomologie
	\'etale des sch\'emas quasi-excellents}, Expos\'e XVI,
	Ast\'erisque 363--364 (2014), 301--349.
	
	\bibitem{Sai17a}
	T. Saito, \emph{The characteristic cycle and the singular support of a constructible sheaf},
	Invent. Math. 207 (2017), 597--695,
	\href{https://doi.org/10.1007/s00222-016-0675-3}{doi:10.1007/s00222-016-0675-3}.
	
%
	
	\bibitem{Sai21}
	T. Saito,
	\emph{Characteristic cycles and the conductor of direct image},
	J. Amer. Math. Soc. 34 (2021), 369--410,
	\href{https://doi.org/10.1090/jams/959}{doi:10.1090/jams/959}.
	
%
	\bibitem{Stacks}
	The Stacks Project Authors, \emph{The Stacks Project},
	\url{https://stacks.math.columbia.edu}.

    
	\bibitem{UYZ}
	N. Umezaki, E. Yang, and Y. Zhao,
	\emph{Characteristic class and the $\varepsilon$-factor of an \'etale sheaf},
	Trans. Amer. Math. Soc. 373 (2020), 6887--6927,
	\href{https://doi.org/10.1090/tran/8187}{doi:10.1090/tran/8187}.
	
	
	
	\bibitem{YZ18}
	E. Yang and Y. Zhao, \emph{On the relative twist formula of $\ell$-adic sheaves},
	Acta Math. Sin. (Engl. Ser.) 37 (2021), 73--94,
	\href{https://doi.org/10.1007/s10114-019-8356-y}{doi:10.1007/s10114-019-8356-y}.
	

	\bibitem{YZ22}
	E. Yang and Y. Zhao, \emph{Cohomological Milnor formula and Saito's conjecture on characteristic classes},
	Invent. Math. 240 (2025), 123--191,
	\href{https://doi.org/10.1007/s00222-025-01319-y}{doi:10.1007/s00222-025-01319-y}.

\end{thebibliography}
\end{document}